\documentclass[10pt,oneside]{amsart}

\usepackage[
paper=a4paper,
text={138mm,195mm},centering
]{geometry}

\usepackage{amssymb,amsxtra}
\usepackage{float}
\usepackage[all]{xy}
\usepackage{pb-diagram,pb-xy}
\usepackage{graphicx,color,float}
\usepackage[bookmarks]{hyperref}
\usepackage{pinlabel}

\iftrue
\makeatletter
\def\@settitle{%
  \vspace*{-20pt}
  \begin{flushleft}%
    \baselineskip14\p@\relax
    \normalfont\bfseries\LARGE
    \@title
  \end{flushleft}%
}
\def\@setauthors{%
  \begingroup
  \def\thanks{\protect\thanks@warning}%
  \trivlist
  \raggedright
  \large \@topsep30\p@\relax
  \advance\@topsep by -\baselineskip
  \item\relax
  \author@andify\authors
  \def\\{\protect\linebreak}%
  \authors
  \ifx\@empty\contribs
  \else
    ,\penalty-3 \space \@setcontribs
    \@closetoccontribs
  \fi
  \normalfont
  \@setaddresses
  \endtrivlist
  \endgroup
}
\def\@setaddresses{\par
  \nobreak \begingroup
  \small
  \def\author##1{\nobreak\addvspace\smallskipamount}%
  \def\\{\unskip, \ignorespaces}%
  \interlinepenalty\@M
  \def\address##1##2{\begingroup
    \par\addvspace\bigskipamount\noindent
    \@ifnotempty{##1}{(\ignorespaces##1\unskip) }%
    {\ignorespaces##2}\par\endgroup}%
  \def\curraddr##1##2{\begingroup
    \@ifnotempty{##2}{\nobreak\noindent\curraddrname
      \@ifnotempty{##1}{, \ignorespaces##1\unskip}\/:\space
      ##2\par}\endgroup}%
  \def\email##1##2{\begingroup
    \@ifnotempty{##2}{\nobreak\noindent E-mail address%
      \@ifnotempty{##1}{, \ignorespaces##1\unskip}\/:\space
      \ttfamily##2\par}\endgroup}%
  \def\urladdr##1##2{\begingroup
    \def~{\char`\~}%
    \@ifnotempty{##2}{\nobreak\noindent\urladdrname
      \@ifnotempty{##1}{, \ignorespaces##1\unskip}\/:\space
      \ttfamily##2\par}\endgroup}%
  \addresses
  \endgroup
  \global\let\addresses=\@empty
}
\def\@setabstracta{%
    \ifvoid\abstractbox
  \else
    \skip@25\p@ \advance\skip@-\lastskip
    \advance\skip@-\baselineskip \vskip\skip@
    \box\abstractbox
    \prevdepth\z@ 
    \vskip-10pt
  \fi
}
\renewenvironment{abstract}{%
  \ifx\maketitle\relax
    \ClassWarning{\@classname}{Abstract should precede
      \protect\maketitle\space in AMS document classes; reported}%
  \fi
  \global\setbox\abstractbox=\vtop \bgroup
    \normalfont\small
    \list{}{\labelwidth\z@
      \leftmargin0pc \rightmargin\leftmargin
      \listparindent\normalparindent \itemindent\z@
      \parsep\z@ \@plus\p@
      
    }%
    \item[\hskip\labelsep\bfseries\abstractname.]%
}{%
  \endlist\egroup
  \ifx\@setabstract\relax \@setabstracta \fi
}
\def\section{\@startsection{section}{1}%
  \z@{-1.2\linespacing\@plus-.5\linespacing}{.8\linespacing}%
  {\normalfont\bfseries\Large}}
\def\subsection{\@startsection{subsection}{2}%
  \z@{-.8\linespacing\@plus-.3\linespacing}{.3\linespacing\@plus.2\linespacing}%
  {\normalfont\bfseries}}
\def\subsubsection{\@startsection{subsection}{3}%
  \z@{.7\linespacing\@plus.2\linespacing}{-1.5ex}%
  {\normalfont\itshape}}
\def\@secnumfont{\bfseries}
\makeatother
\fi 

\def\to{\mathchoice{\longrightarrow}{\rightarrow}{\rightarrow}{\rightarrow}}
\makeatletter
\newcommand{\shortxra}[2][]{\ext@arrow 0359\rightarrowfill@{#1}{#2}}
\def\longrightarrowfill@{\arrowfill@\relbar\relbar\longrightarrow}
\newcommand{\longxra}[2][]{\ext@arrow 0359\longrightarrowfill@{#1}{#2}}
\renewcommand{\xrightarrow}[2][]{\mathchoice{\longxra[#1]{#2}}%
  {\shortxra[#1]{#2}}{\shortxra[#1]{#2}}{\shortxra[#1]{#2}}}
\makeatother

\def\otimesover#1{\mathbin{\mathop{\otimes}_{#1}}}

\makeatletter
\def\Nopagebreak{\@nobreaktrue\nopagebreak}
\makeatother

\theoremstyle{plain}
\newtheorem{theorem}{Theorem}[section]
\newtheorem{proposition}[theorem]{Proposition}
\newtheorem{corollary}[theorem]{Corollary}
\newtheorem{lemma}[theorem]{Lemma}

\theoremstyle{definition}
\newtheorem{definition}[theorem]{Definition}
\newtheorem{question}[theorem]{Question}

\newtheorem{remark}[theorem]{Remark}

\def\Z{\mathbb{Z}}
\def\Q{\mathbb{Q}}
\def\R{\mathbb{R}}
\def\C{\mathbb{C}}

\def\N{\mathcal{N}}

\def\cP{\mathcal{P}}

\def\T{\mathcal{T}}

\def\TSL{\mathcal{T}^\mathrm{SL}}

\def\Ker{\operatorname{Ker}}

\def\Hom{\operatorname{Hom}}
\def\Tor{\operatorname{Tor}}

\def\sign{\operatorname{sign}}
\def\rank{\operatorname{rank}}

\def\BH{\operatorname{BH}}

\def\lk{\operatorname{lk}}

\def\ldim{\dim^{(2)}}
\def\lsign{\sign^{(2)}}

\def\rhot{\rho^{(2)}}
\def\wt{\widetilde}

\def\spin{\text{spin}}

\def\Wh{\operatorname{Wh}}

\newcommand{\eps}{\varepsilon}

\let\oldsharp=\# \def\#{\mathbin{\oldsharp}}
\def\bigconnsum{\mathop{\#}\limits}

\begin{document}

\title%
[Covering link calculus and the bipolar filtration]
{Covering link calculus and the bipolar filtration of topologically slice
  links}

\author{Jae Choon Cha}
\address{
  Department of Mathematics\\
  POSTECH\\
  Pohang 790--784\\
  Republic of Korea\\
  and\linebreak
  School of Mathematics\\
  Korea Institute for Advanced Study \\
  Seoul 130--722\\
  Republic of Korea
}
\email{jccha@postech.ac.kr}

\author{Mark Powell}
\address{
  Department of Mathematics\\
  Indiana University \\
  Bloomington, IN 47405\\
  USA
}
\email{macp@indiana.edu}

\def\subjclassname{\textup{2010} Mathematics Subject Classification}
\expandafter\let\csname subjclassname@1991\endcsname=\subjclassname
\expandafter\let\csname subjclassname@2000\endcsname=\subjclassname


\begin{abstract}
  The bipolar filtration introduced by T. Cochran, S. Harvey, and
  P. Horn is a framework for the study of smooth concordance of
  topologically slice knots and links.  It is known that there are
  topologically slice 1-bipolar knots which are not 2-bipolar. For
  knots, this is the highest known level at which the filtration does
  not stabilize.  For the case of links with two or more components,
  we prove that the filtration does not stabilize at any level: for
  any $n$, there are topologically slice links which are $n$-bipolar
  but not $(n+1)$-bipolar.  In the proof we describe an explicit
  geometric construction which raises the bipolar height of certain
  links exactly by one.  We show this using the covering link
  calculus.  Furthermore we discover that the bipolar filtration of
  the group of topologically slice string links modulo smooth
  concordance has a rich algebraic structure.
\end{abstract}

\maketitle

\section{Introduction}

Since the stunning work of S. Donaldson and M. Freedman in the early
1980s, the smoothing of topological 4-manifolds has been a central
subject in low dimensional topology.  While there have been
significant advances in this area, we are still far from having a
complete understanding.  An experimental lab for the study of this
difference in categories is the comparison between smooth and
topological concordance of knots and links in~$S^3$.  In this context
various techniques from Donaldson and Seiberg-Witten theory, to more
recent tools arising from Heegaard Floer and Khovanov homology, have
been used to give exciting results.  In particular, since A. Casson
observed that Donaldson's work could be applied to show that there are
topologically slice knots which are not smoothly slice, smooth concordance
of topologically slice knots and links has been studied extensively.

In order to understand the structure of topologically slice knots and
links, T. Cochran, S. Harvey, and P. Horn introduced a framework for
the study of smooth concordance in their recent
paper~\cite{Cochran-Harvey-Horn:2012-1}.  This is an intriguing
attempt at describing a global picture of the world of topologically
slice links.  They defined the notion of \emph{$n$-bipolarity} of
knots and links, as an approximation to honest slicing whose accuracy
is measured in terms of the derived series of the fundamental group of
slice disk complements in certain positive/negative definite
4-manifolds.  (For a precise definition, see Definition
\ref{definition:homotopy-n-positon} or
\cite[Definition~2.1]{Cochran-Harvey-Horn:2012-1}.)  This refines
Cochran-Orr-Teichner's $(n)$-solvable filtration which organizes the
study of topological concordance.  Note that a topologically slice
link is $(n)$-solvable for all $n$, so that the solvable filtration
contains no information about the difference in categories.

For each $m$, the collection $\T_n$ of concordance classes of
\emph{topologically slice} $n$-bipolar links with $m$ components form
a filtration
\[
\{[\text{unlink}]\} \subset \cdots \subset \T_2 \subset \T_1\subset
\T_0 \subset \T=\frac{\{\text{topologically slice $m$-component
  links}\}}{\text{concordance}}.
\]
A smoothly slice link lies in $\T_n$ for all $n$.  An
important feature which helps to justify this theory is that
previously known smooth concordance obstructions are
related to the low level terms.  In particular, for 1-bipolar knots,
the following obstructions to being slice vanish: the $\tau$-invariant
and the $\epsilon$-invariant from Heegaard Floer Knot homology, the
$s$-invariant from the Khovanov homology (and consequently the
Thurston-Bennequin invariant), and the Heegaard Floer correction term
$d$-invariant of $\pm1$-surgery manifolds and prime power fold cyclic
branched covers, as well as gauge theoretic obstructions derived from
work of Fintushel-Stern, Furuta, Endo, and Kirk-Hedden.  For more
details the reader is referred to~\cite{Cochran-Harvey-Horn:2012-1}.

A fundamental question is whether the filtration is non-trivial at
every level.  This is difficult to answer because, as discussed above,
known smooth invariants vanish in the higher terms of the filtration.
The best previously known result, due
to~\cite{Cochran-Harvey-Horn:2012-1,Cochran-Horn:2012-1}, is that
\[
\T_2 \subsetneq \T_1 \subsetneq \T_0 \subsetneq \T
\]
for knots.  That is, for $n = -1, 0, 1$, there are topologically slice
$n$-bipolar knots which are not $(n+1)$-bipolar, where for convenience
``topologically slice $(-1)$-bipolar'' is understood as
``topologically slice.''  Consequently, for links with any given
number of components, the filtration is non-trivial at level $n$ for
each $n\le 1$.  The knots of Hedden-Livingston-Ruberman
from~\cite{Hedden-Livingston-Ruberman:2010-01}, which were the first
examples of topologically slice knots which are not smoothly
concordant to knots with Alexander polynomial one, are also nontrivial
in~$\T/\T_0$.

The main result of this paper is to show the non-triviality of the
filtration at every level for links.

\begin{theorem}\label{theorem:main_theorem}
  For any $m\ge2$ and $n\ge 0$, there are topologically slice
  $m$-component links which are $n$-bipolar but not $(n+1)$-bipolar.
\end{theorem}

We remark that, for $n\ge 1$, the links which we exhibit in the proof
of Theorem~\ref{theorem:main_theorem} have unknotted components.

In order to prove Theorem~\ref{theorem:main_theorem}, we introduce the
notion of a \emph{$\Z_{(p)}$-homology $n$-bipolar link} (Definition
\ref{definition:homology-n-positon}) and employ the method of
\emph{covering link calculus} following the formulation
in~\cite{Cha-Kim:2008-1}.  An $n$-bipolar link is $\Z_{(p)}$-homology $n$-bipolar for all primes $p$.  The key ingredient (Theorem
\ref{theorem:covering-positon-theorem}), which we call the
\emph{Covering Positon/Negaton Theorem}, is that a covering link of a
$\Z_{(p)}$-homology $n$-bipolar link of height $k$ is
$\Z_{(p)}$-homology $(n-k)$-bipolar.

An interesting aspect of the proof of
Theorem~\ref{theorem:main_theorem} is that our examples are described
explicitly using a geometric operation, which pushes a link into a
higher level of the bipolar filtration.  For this purpose it is useful
to consider the notion of the \emph{bipolar height} of a link, which
is defined by
\[
\BH(L)=\max\{n \mid \text{$L$ is $n$-bipolar}\}.
\]
Using the Covering Positon/Negaton Theorem and the calculus of
covering links, we show that for certain class of links our geometric
operation raises the bipolar height (and its $\Z_{(p)}$-homology
analogue) precisely by one.  See
Definitions~\ref{definition:bipolar-height},
\ref{definition:bipolar-height-plus} and \ref{definition:C(-)} and
Theorem~\ref{theorem:bipolar-height-raising} for more details. This
enables us to push the rich structure of $\T_0$ modulo $\T_1$, which
was revealed by Cochran and Horn for the knot case in
\cite{Cochran-Horn:2012-1} using $d$-invariants, to an arbitrarily
high level of the bipolar filtration of links.

For links, the concordance classes merely form a set, of which
$\{\T_n\}$ is a filtration by subsets, since the connected sum
operation is not well defined.  The standard approach to generalize
the algebraic structure of the knot concordance group is to consider
\emph{string links}; the concordance classes of string links with $m$
components form a group, and it turns out that the classes of
topologically slice string links with $n$-bipolar closures form a
normal subgroup, which we denote by~$\TSL_n(m)$.  We discuss more
details at the beginning of Section~\ref{Section:string-links}.  This
bipolar filtration of topologically slice string links has a rich
algebraic structure:

\begin{theorem}\label{theorem:string-links-infinite-rank}
  For any $n \ge 0$ and for any $m \ge 2$, the quotient
  $\TSL_n(m)/\TSL_{n+1}(m)$ contains a subgroup whose abelianization
  is of infinite rank.
\end{theorem}

Once again, for $n\ge 1$ these subgroups are generated by
string links with unknotted components.

In the proof of Theorem~\ref{theorem:string-links-infinite-rank}, we
introduce a string link version of covering link construction which
behaves nicely with respect to the group structure (see for instance
Definition~\ref{definition:covering-string-link} and
Lemma~\ref{lemma:covering-string-link-and-product}).  We employ a more
involved calculus of covering string links than that used in the
ordinary links case.  So far as the authors know, this is the first
use of covering link calculus to investigate the structure of the
string link concordance group.  We anticipate that our method for
string links will be useful for further applications.

The paper is organized as follows: in Section
\ref{section:definitions_and_basic_results} we give the definition of a
$\Z_{(p)}$-homology $n$-bipolar link, and prove some basic results about
this notion.  In Section~\ref{section:covering_link_calculus} we
discuss the covering link calculus, and prove the Covering
Positon/Negaton theorem.  We also prove some preliminary results on
Bing doubles, using covering link calculus results for Bing doubles
from the literature.  In Section~\ref{section:main-examples} we prove
the bipolar height raising result, exhibit our main examples, and thus
prove Theorem~\ref{theorem:main_theorem}.
Section~\ref{Section:string-links} considers string links and proves
Theorem \ref{theorem:string-links-infinite-rank}.

Appendices~\ref{section:signature-and-0-positivity}
and~\ref{section:rho-invariant-and-positivity} give results on the use
of Levine-Tristram signatures and Von Neumann $\rho$-invariants
respectively to obstruct a link being $\Z_{(p)}$-homology $n$-bipolar.

\subsubsection*{Conventions}
All manifolds are smooth and submanifolds are smoothly embedded, with
the exception of the locally flat disks which are tacitly referred to
when we claim that a link is topologically slice.  If not specified,
then homology groups are with $\Z$ coefficients by default.  The
letter $p$ always denotes a prime number.

\subsubsection*{Acknowledgements}
The authors would like to thank Tim Cochran and Shelly Harvey for
interesting conversations about their filtration.  The authors also
thank an anonymous referee for helpful comments.  The first author was
partially supported by NRF grants 2010--0011629, 2013067043, and
2013053914.

\section{Homology $n$-positive, negative, and bipolar links}
\label{section:definitions_and_basic_results}

An (oriented) $m$-component link $L = L_1 \, \sqcup \dots \sqcup \,
L_m$ is a smooth disjoint embedding of $m$ (oriented) copies of $S^1$
into $S^3$.  Two links $L$ and $L'$ are said to be \emph{smoothly
  \textup{(}resp.\ topologically\textup{)} concordant} if there are $m$ disjoint
smoothly (resp.\ locally flat) embedded annuli $C_i$ in $S^3 \times I$
with $\partial C_i = L_i \times \{0\} \sqcup -L'_i \times \{1\}$.
Here $-L'$ is the mirror image of $L'$ with the string orientation
reversed.  A link which is concordant to the unlink is said to be
\emph{slice}.

In~\cite{Cochran-Harvey-Horn:2012-1}, Cochran, Harvey, and Horn
introduced the notion of $n$-positivity, negativity, and bipolarity.

\begin{definition}[{\cite[Definition~2.2]{Cochran-Harvey-Horn:2012-1}}]
  \label{definition:homotopy-n-positon}
  An $m$-component link $L$ in $S^3$ is \emph{$n$-positive} if $S^3$
  bounds a connected 4-manifold $V$ satisfying the following:
  \begin{enumerate}
  \item $\pi_1(V) \cong 0$.
  \item There are disjoint smoothly embedded 2-disks
    $\Delta_1,\ldots,\Delta_m$ in $V$ with $\bigsqcup\partial\Delta_i
    = L$.
  \item The intersection pairing on $H_2(V)$ is
    positive definite and there are disjointly embedded connected
    surfaces $S_j$ in $V-\bigsqcup \Delta_i$ which generate
    $H_2(V)$ and satisfy $\pi_1(S_j)\subset
    \pi_1(V-\bigsqcup \Delta_i)^{(n)}$.
  \end{enumerate}
  We call $V$ as above an \emph{$n$-positon} for $L$
  with slicing disks~$\Delta_i$.  An~\emph{$n$-negative
    link} and an \emph{$n$-negaton} are defined by
  replacing ``positive definite'' with ``negative definite''.  A link
  $L$ is called \emph{$n$-bipolar} if $L$ is both $n$-positive and
  $n$-negative.
\end{definition}

For our purpose the following homology analogue of
Definition~\ref{definition:homotopy-n-positon} provides an optimum
setting.

Throughout this paper $p$ denotes a fixed prime.  Therefore, as in
Definition~\ref{definition:Z-localized-at-p-coeff-derived-series}, we
often omit the $p$ from the notation.  In the statements of theorems
also we omit that the theorem holds for all primes~$p$.  At the end of
the paper, in the proof of
Theorem~\ref{theorem:2-component-string-links-infinite-rank}, we need
to restrict to $p=2$ for $d$-invariant computational reasons, but
until then $p$ can be any fixed prime.

We denote the localization of $\Z$ at $p$ by $\Z_{(p)}=\{a/b \in
\Q\mid p\nmid b \}$, and the ring $\Z/p\Z$ of mod $p$ residue classes
by~$\Z_p$.  Note that a manifold is a $\Z_p$-homology sphere if and
only if it is a $\Z_{(p)}$-homology sphere.

\begin{definition}
  \label{definition:Z-localized-at-p-coeff-derived-series}
  We define the \emph{$\Z_{(p)}$-coefficient derived series} $\{\cP^n
  G\}$ of a group $G$ as follows: $\cP^0 G := G$,
  \[
  \cP^{n+1} G := \Ker \Big( \cP^n G \rightarrow
  \frac{\cP^n G}{[\cP^n G,\cP^n G]} \rightarrow
  \frac{\cP^n G}{[\cP^n G,\cP^n G]}\otimesover{\Z} \Z_{(p)} \cong
  H_1(\cP^n G;\Z_{(p)})\Big)
  \]
\end{definition}

From the definition it can be seen that $\cP^n G$ is a normal subgroup
of $G$ for any~$n$.

We remark that a more general case of the mixed-coefficient derived
series was defined in~\cite{Cha:2010-01}.  It should also be noted
that our $\Z_{(p)}$-coefficient derived series is not equal to the
$\Z_p$-coefficient analogue, namely the series obtained by replacing
$\Z_{(p)}$ with~$\Z_p$, which appeared in the literature as well;
$\cP^n G/\cP^{n+1} G$ is the maximal abelian quotient of $\cP^n G$
which has no torsion coprime to~$p$, while the corresponding quotient
of the $\Z_p$-coefficient analogue is the maximal abelian quotient
which is a $\Z_p$-vector space.

\begin{definition}
  [$\Z_{(p)}$-homology $n$-positivity, negativity, and bipolarity]
  \label{definition:homology-n-positon}
  Suppose $L$ is an $m$-component link in a $\Z_{(p)}$-homology
  3-sphere~$Y$.  We say that $(Y,L)$ is \emph{$\Z_{(p)}$-homology
    $n$-positive} if $Y$ bounds a connected 4-manifold $V$ satisfying
  the following:
  \begin{enumerate}
  \item $H_1(V;\Z_{(p)})=0$.
  \item There are disjointly embedded 2-disks
    $\Delta_1,\ldots,\Delta_m$ in $V$ with $\bigsqcup\partial\Delta_i
    = L$.
  \item The intersection pairing on $H_2(V)/\text{torsion}$ is
    positive definite and there are disjointly embedded connected
    surfaces $S_j$ in $V-\bigsqcup \Delta_i$ which generate
    $H_2(V)/\text{torsion}$ and satisfy $\pi_1(S_j)\subset \cP^n
    \pi_1(V-\bigsqcup \Delta_i)$.
  \end{enumerate}

  We call $V$ as above a \emph{$\Z_{(p)}$-homology $n$-positon} for
  $(Y,L)$ with slicing disks~$\Delta_i$.  A~\emph{$\Z_{(p)}$-homology
    $n$-negative link} and a \emph{$\Z_{(p)}$-homology $n$-negaton}
  are defined by replacing ``positive definite'' with ``negative
  definite.''  We say that $(Y,L)$ \emph{$\Z_{(p)}$-homology
    $n$-bipolar} if $(Y,L)$ is both $\Z_{(p)}$-homology $n$-positive
  and $\Z_{(p)}$-homology $n$-negative.

  For a commutative ring $R$ (e.g.\ $R=\Z$ or $\Q$), the
  \emph{$R$-homology} analogues are defined by replacing $\Z_{(p)}$
  with $R$ in the above definitions.
\end{definition}

We note that whether a link is $\Z_{(p)}$-homology $n$-positive
depends on the choice of the ambient space; it is conceivable that,
for example, even when $(Y,L)$ is not $\Z_{(p)}$-homology
$n$-positive, $(Y\#Y',L)$ could be.  Nonetheless, when the choice of
the ambient space $Y$ is clearly understood, we often say that $L$ is
$\Z_{(p)}$-homology $n$-positive, and similarly for the
negative/bipolar case.

When we need to distinguish $n$-positive (resp.\ negative, bipolar)
links in Definition~\ref{definition:homotopy-n-positon} explicitly
from the $\Z_{(p)}$-homology case in
Definition~\ref{definition:Z-localized-at-p-coeff-derived-series}, we
call the former \emph{homotopy} $n$-positive (resp.\ negative,
bipolar).  It is easy to see that homotopy $n$-positive (resp.\ negative,
bipolar) links are $R$-homology $n$-positive (resp.\ negative,
bipolar) for any~$R$.

\subsection{Zero-framing and 0-positivity/negativity}

In this paper, we need some basic facts on framings of components of links
in rational homology spheres (for example to define branched covers).
For the reader's convenience we discuss these in some detail, focusing on
the $0$-positive/negative case.

The following lemma gives the basic homological properties of a
$\Z_{(p)}$-homology 0-positon (or negaton).  Since an $n$-positon, for
$n > 0$, is in particular a $0$-positon, this will suffice.

\begin{lemma}
  \label{lemma:basic-properties-of-Z_p-positon}
  Suppose $V$ is a $\Z_{(p)}$-homology $0$-positon \textup{(}or
  negaton\textup{)} for an $m$-component link $L$ with slicing
  disks~$\Delta_i$.  Then the following hold:
  \begin{enumerate}
  \item\label{lemma-case:H1-of-slice-disk-complement-in-positon} The
    first homology $H_1(V-\bigsqcup\Delta_i;\Z_{(p)})$ is a free
    $\Z_{(p)}$-module of rank $m$ generated by the meridians of~$L$.
  \item\label{lemma-case:H2-of-positon} The inclusion induces an
    isomorphism of $H_2(V-\bigsqcup\Delta_i;\Z_{(p)})$ onto
    $H_2(V;\Z_{(p)})$, which is a free $\Z_{(p)}$-module generated by
    the surfaces $S_j$ in
    Definition~\ref{definition:homology-n-positon}.
  \end{enumerate}
\end{lemma}

The following elementary observations are useful: commutative
localizations are flat, so:
\[
H_*(-;\Z_{(p)})\cong H_*(-)\otimes \Z_{(p)},
\]
and for a finitely generated abelian group $A$, we have
\begin{align*}
  A\otimes \Z_{(p)} \cong (\Z_{(p)})^d &\Longleftrightarrow A\cong
  \Z^d \oplus \text{(torsion coprime to $p$)}
  \\
  & \Longleftrightarrow A\otimes \Q \cong \Q^d \text{ and } A\otimes
  \Z_{p^a} \cong (\Z_{p^a})^d \text{ for any $a$}
\end{align*}
In addition, if these equivalent conditions hold, then for elements
$x_1,\ldots,x_d\in A$, $\{x_i \otimes 1\}$ is a basis of
$A\otimes\Z_{(p)}$ if and only if $\{x_i \otimes 1\}$ is a basis of
$A\otimes\Z_{p^a}$ and $\{x_i \otimes 1\}$ is a basis of $A\otimes\Q$.
This gives us the following consequences of
Lemma~\ref{lemma:basic-properties-of-Z_p-positon}: the torsion part of
$H_1(V-\bigsqcup \Delta_i)$ has order coprime to $p$, and for
$R=\Z_{p^a}$ and $\Q$, $H_1(V-\bigsqcup \Delta_i;R)$ is isomorphic to
the free $R$-module $R^m$ generated by the meridians.  Similar
conclusions hold for $H_2(V-\bigsqcup \Delta_i; -)$ and $H_2(V; -)$.

\begin{proof}[Proof of Lemma~\ref{lemma:basic-properties-of-Z_p-positon}]

  First we claim that $H_2(V)$ has no $p$-torsion.  To see this,
  note that since $\partial V$ is a $\Z_p$-homology sphere and
  $H_1(V;\Z_{(p)})=0$ implies $H_1(V;\Z_p)=0$, we have
  $H_3(V;\Z_p)\cong H^1(V,\partial V;\Z_p)\cong H^1(V;\Z_p)=0$.
  By the universal coefficient theorem, $H_3(V;\Z_p)$ surjects
  onto $\Tor(H_2(V),\Z_p)$.  The vanishing of this latter group
  implies the claim.

  From the above claim it follows that the quotient map $H_2(V) \to
  H_2(V)/\text{torsion}$ induces an isomorphism $H_2(V;\Z_{(p)}) =
  H_2(V)\otimes \Z_{(p)} \cong (H_2(V)/\text{torsion})\otimes
  \Z_{(p)}$.  Let $S_j$ be the surfaces given in
  Definition~\ref{definition:homology-n-positon}.  Since $S_j\subset
  V-\bigsqcup \Delta_i$ and the classes of $S_j$ span
  $H_2(V)/\text{torsion}$, it follows that $H_2(V-\bigsqcup
  \Delta_i;\Z_{(p)}) \to H_2(V;\Z_{(p)})$ is surjective.  From the
  long exact sequence for $(V,V-\bigsqcup\Delta_i)$ and excision it
  follows that
  \[\textstyle
  H_1(V-\bigsqcup\Delta_i;\Z_{(p)}) \cong
  H_2(V,V-\bigsqcup\Delta_i;\Z_{(p)})\cong \bigoplus_i
  H_2(\Delta_i\times(D^2,S^1);\Z_{(p)})\cong (\Z_{(p)})^m,
  \]
  generated by the meridians of the $\Delta_i$.  This
  shows~(\ref{lemma-case:H1-of-slice-disk-complement-in-positon}).

  Now, for (\ref{lemma-case:H2-of-positon}), consider the
  $\Z_{(p)}$-coefficient Mayer-Vietoris sequence for $V \simeq (V-\bigsqcup\Delta_i) \cup
  (\text{$2$-cells})$, where the 2-cells are attached along each
  meridian. (Alternatively, consider the long exact sequence for
  $(V,V-\bigsqcup\Delta_i)$.)  By
  (\ref{lemma-case:H1-of-slice-disk-complement-in-positon}), the
  desired conclusion easily follows.
\end{proof}

A framing of a submanifold is a choice of trivialization of its normal
bundle.  Generalizing the case of knots in $S^3$, a framing $f$ of a
knot $K$ in a rational homology sphere $Y$ is called the
\emph{zero-framing} if the rational-valued linking number of $K$ with
its longitude ($=$ a parallel copy) taken along $f$ is zero.  We state
some facts on the zero-framing.  For the proofs, see
\cite[Chapter~2]{Cha:2003-1}, \cite[Section~3]{Cha-Ko:2000-1}.
\begin{enumerate}
\item A knot $K$ in a rational homology sphere has a zero framing if
  and only if the $\Q/\Z$-valued self-linking of $K$ vanishes. This
  follows from \cite[Theorem~2.6~(2)]{Cha:2003-1}, noting that a knot
  admits a generalized Seifert surface with respect to some framing if
  and only if that framing is the zero framing, by
  \cite[Lemma~3.1]{Cha-Ko:2000-1}.
\item A framing $f$ on $K$ in $Y$ is a zero-framing if and only if
  there is a map $g\colon Y-K \to S^1$ under which the longitude (= a
  parallel copy) taken along $f$ is sent to a null-homotopic curve and
  the image of a meridian $\mu$ of $K$ is sent to an essential curve.
  Furthermore, $g_*([\mu])\in \Z=\pi_1(S^1)$ is coprime to $p$ if $Y$
  is a $\Z_{(p)}$-homology 3-sphere.  This follows from elementary
  obstruction theory and the same observation as above, namely that
  zero-framing is equivalent to the existence of a generalized Seifert
  surface.
\end{enumerate}
From the above it follows that if a knot $K$ in a $\Z_{(p)}$-homology
3-sphere $Y$ admits a zero-framing and $d$ is a power of $p$, then the
$d$-fold cyclic branched cover of $Y$ along $K$ is defined.  For, the
above map $g$ induces $\phi\colon \pi_1(Y-K)\to \Z \to \Z_d$ which
gives a $d$-fold regular cover of $Y-K$.  Under $\phi$ the meridian of
$K$ is sent to a generator of $\Z_d$, since $d$ is a power of $p$.  It
follows that one can glue the $d$-fold cover of $Y-K$ with the
standard branched covering of $K\times (D^2,0)$ along the zero-framing
to obtain the desired branched cover of $Y$ along~$K$.

\begin{lemma}
  \label{lemma:positivity-zero-framing}
  A $\Z_{(p)}$-homology $0$-positive \textup{(}or negative\textup{)}
  knot $K$ in a $\Z_{(p)}$-homology sphere admits a zero-framing.
\end{lemma}

\begin{proof}
  Suppose $V$ is a $\Z_{(p)}$-homology $0$-positon for $L$ with a
  slicing disk~$\Delta$.  By appealing to
  Lemma~\ref{lemma:basic-properties-of-Z_p-positon}~(\ref{lemma-case:H1-of-slice-disk-complement-in-positon}),
  $H_1(V-\Delta)/\text{torsion}$ is isomorphic to an infinite cyclic
  group where the meridian of $K$ is a nonzero power of the generator.
  This gives rise to a map $V-\Delta\to S^1$ by elementary obstruction
  theory.  Let $g \colon Y-K\to S^1$ be its restriction.  The
  longitude of $K$ taken along the unique framing of the normal bundle
  of $\Delta$ is null-homotopic in $V-\Delta$, and consequently its
  image under $g$ is null-homotopic in~$S^1$.
\end{proof}

\subsection{Basic operations on links}

We state some observations on how positivity (resp.\ negativity,
bipolarity) are affected by certain basic operations for links as the
following theorem.  Throughout this paper, $I=[0,1]$.  For two links
$L\subset Y$ and $L'\subset Y'$, their \emph{split union} is defined
to be the link $L\sqcup L' \subset Y\# Y'$, where the connected sum is
formed by choosing 3-balls disjoint to the links.  For a link
$L\subset Y$ and an embedding $\gamma\colon I\times I \to Y$ such that
$\gamma(0\times I)$ and $\gamma(1\times I)$ are contained in distinct
components of $L$ and $\gamma((0,1)\times I)$ is disjoint to $L$, a
new link $L'$ is obtained by smoothing the corners of
$(L-\gamma(\{0,1\}\times I)) \cup \gamma(I\times\{0,1\})$.  We say
$L'$ is obtained from $L$ by \emph{band sum} of components.  When $L$
is oriented, we assume that the orientation of $\gamma(\{0,1\}\times
I)$ is opposite to the orientations of $L$ so that $L'$ is oriented
naturally.  For a submanifold $J$, we denote the normal bundle
by~$\nu(J)$.

\begin{theorem}
  \label{theorem:basic-construction-and-positivity}
  The following statements and their $\Z_{(p)}$-homology analogues
  hold.  In addition, the statements also hold if the words
  ``positive'' and ``negative'' are interchanged.  Similarly the
  statements hold if both these words are replaced by ``bipolar'',
  except for~(\ref{theorem-case:positivity-crossing-change}).
  \begin{enumerate}
  \item\label{theorem-case:positivity-mirror-image}
    \textup{(Mirror image)} A link is $n$-positive if and only if
    its mirror image is $n$-negative.
  \item\label{theorem-case:positivity-changing-string-orientation}
    \textup{(String orientation change)} A link if $n$-positive if
    and only if the link obtained by reversing orientation of any of
    the components is $n$-positive.
  \item\label{theorem-case:positivity-split-union} \textup{(Split
      union)} The split union of two $n$-positive link is $n$-positive.
  \item\label{theorem-case:positivity-sublink} \textup{(Sublink)}
    A sublink of an $n$-positive link is
    $n$-positive.
  \item\label{theorem-case:positivity-band-sum} \textup{(Band
      sum)} A link obtained by band sum of components from an
    $n$-positive link is $n$-positive.
  \item\label{theorem-case:positivity-generalize-satellite}
    \textup{(Generalized cabling and doubling)} Suppose $L_0$ is a
    link in the standard $S^1\times D^2 \subset S^3$ which is slice in
    the 4-ball, and $L$ is an $n$-positive
    link with a component~$J$.  Then the link obtained from $L$ by
    replacing $(\nu(J), J)$ with $(S^1\times D^2, L_0)$ along the zero
    framing is $n$-positive.
  \item\label{theorem-case:blow-down} \textup{(Blow-down)} If $L$
    is an $n$-positive $m$-component link, then the $(m-1)$-component
    link in the $(\pm1)$-framed \textup{(}with respect to the
    zero-framing\textup{)} surgery along a component of $L$ is
    $n$-positive.
  \item\label{theorem-case:positivity-slice-link}
    \textup{(Concordance)} A link concordant to an $n$-positive
    link is $n$-positive.  A slice link is $n$-bipolar for any~$n$.
  \item\label{theorem-case:positivity-crossing-change}
    \textup{(Crossing change)} If $L$ is transformed to a
    $0$-positive link by changing some positive crossings involving
    the same components to negative crossings, then $L$ is
    $0$-positive.
  \end{enumerate}
\end{theorem}

We remark that due to
Theorem~\ref{theorem:basic-construction-and-positivity}~(\ref{theorem-case:positivity-changing-string-orientation})
one may view knot and links as unoriented for the study of
$n$-positivity (resp.\ negativity, bipolarity).

From Theorem~\ref{theorem:basic-construction-and-positivity}
(\ref{theorem-case:positivity-split-union}) and
(\ref{theorem-case:positivity-band-sum}) above, it follows that
\emph{any} connected sum of $n$-positive (resp.\ negative, bipolar)
links is $n$-positive (resp.\ negative, bipolar).  Moreover, by
Theorem
\ref{theorem:basic-construction-and-positivity}~(\ref{theorem-case:positivity-changing-string-orientation}),
by changing the orientation on one component before band summing, we
can consider band sums which do not respect the orientation of a
component.  We have $\Z_{(p)}$-homology analogues as well.

\begin{proof}
  Suppose $V$ is an $n$-positon for $L$ with slicing disks~$\Delta_i$.

  (\ref{theorem-case:positivity-mirror-image}) The $4$-manifold $V$
  with reversed orientation is an $n$-negaton for the mirror image
  of~$L$.  Note that this makes it sufficient to consider only the
  positive case of the remaining statements: the negative case then
  follows from taking mirror image and the bipolar case from having
  both the positive and negative cases.

  (\ref{theorem-case:positivity-changing-string-orientation}) As
  remarked above, the $n$-positon $V$ and the slicing disks $\Delta_i$
  satisfy Definition \ref{definition:homology-n-positon} and its
  homotopy analogue independently of the orientations of components of
  $L$.

  (\ref{theorem-case:positivity-split-union})~The boundary connected
  sum of $n$-positons (along balls disjoint to links) is an
  $n$-positon of the split link.

  (\ref{theorem-case:positivity-sublink}) The $4$-manifold $V$ with
  appropriate slicing disks forgotten is an $n$-positon for a
  sublink.

  (\ref{theorem-case:positivity-band-sum})~If $L'$ is obtained from
  $L$ by band sum, then $L'$ and $-L$ cobound a disjoint union of
  annuli and a twice-punctured disk in $\partial V\times I$.
  Attaching it to $V$, we obtain an $n$-positon for~$L'$.

  (\ref{theorem-case:positivity-generalize-satellite}) Note that
  $L_0\subset S^1\times D^2 \subset \partial(D^2\times D^2)$ bounds
  slicing disks $D_i$ in $D^2\times D^2$ by the assumption.  If
  $\Delta$ is a slicing disk for a component $J$ of $L$, then
  identifying $\nu(\Delta)$ with $\Delta\times D^2\cong D^2\times D^2$
  and replacing $(\nu(\Delta),\Delta)$ with $(D^2\times D^2, \bigsqcup
  D_i)$ in $V$, we obtain slicing disks for the newly introduced
  components, and with these and the other slicing disks for $L$, $V$ is
  an $n$-positon.

  (\ref{theorem-case:blow-down}) Without loss of generality we can do
  surgery on the first component, say $K_1$, of~$L_1$.  Define a
  4-manifold $W$ to be $V-\nu(\Delta_1)$ with a 2-handle $D^2 \times
  D^2$ attached along the $\Delta_1 \times S^1$ part of the boundary
  of $V-\nu(\Delta_1)$, with an attaching map $S^1 \times D^2 \to
  \Delta_1 \times S^1$ chosen so that the resulting boundary
  3-manifold is that given by performing surgery on $L_1$ using the
  $(\pm 1)$-framing.  The 4-manifold $W$ is an $n$-positon for the new
  link obtained by $(\pm 1)$-surgery along~$L_1$.  For, it can be seen
  that $\pi_1(W)\cong 0$ by the Seifert-Van Kampen theorem, since
  $\pi_1(V-\Delta_1)$ is normally generated by a meridian of~$K_1$ and
  the zero-linking longitude of $K_1$ is null-homotopic in
  $V-\nu(\Delta_1)$.  We also see that $H_2(W) \cong H_2(V)$ by a
  Mayer-Vietoris argument, using
  Lemma~\ref{lemma:basic-properties-of-Z_p-positon}~(\ref{lemma-case:H2-of-positon}).
  (This is true when $V$ is a homotopy $n$-positon and a
  $\Z_{(p)}$-homology $n$-positon.)

  (\ref{theorem-case:positivity-slice-link}) If $L'$ is concordant to
  $L$ via disjoint embedded annuli $C_i$ in $S^3\times I$, then
  $V\cup_{S^3} S^3\times I$ is an $n$-positon for $L'$ with slicing
  disks $\Delta_i \cup C_i$.  For a slice link $L$, the 4-ball with
  slicing disks is an $n$-positon (and negaton) for~$L$ for all $n$.

  (\ref{theorem-case:positivity-crossing-change}) This is shown by
  arguments of \cite[Lemma~3.4]{Cochran-Lickorish:1986-1}.

  The $\Z_{(p)}$-homology analogues are shown by the same arguments.
  We note that when $V$ is a $\Z_{(p)}$-homology $n$-positon in (7),
  we have the following instead of $\pi_1(W)\cong 0$: $H_1(W;\Z_{(p)})
  \cong H_1(V;\Z_{(p)}) \cong 0$ by a Mayer-Vietoris argument, since
  $H_1(V-\Delta_1;\Z_{(p)})$ is isomorphic to $\Z_{(p)}$ generated by
  a meridian by
  Lemma~\ref{lemma:basic-properties-of-Z_p-positon}~(\ref{lemma-case:H1-of-slice-disk-complement-in-positon}).
\end{proof}

We remark that the operations in Theorem
\ref{theorem:basic-construction-and-positivity}~(\ref{theorem-case:positivity-changing-string-orientation}),
(\ref{theorem-case:positivity-split-union}),
(\ref{theorem-case:positivity-sublink}),
(\ref{theorem-case:positivity-band-sum})
and~(\ref{theorem-case:blow-down}) may be used (in various combinations) to
obtain obstructions to links being positive (or negative, bipolar)
from previously known results on knots.  We discuss these results on
knots below.  In the later sections of this paper we will develop more
sophisticated methods of reducing a problem from links to knots.

\subsection{Obstructions to homology 0- and 1-positivity}

Most of the obstructions to knots being $0$- and $1$-positive (resp.\
negative, bipolar) introduced in \cite[Sections 4,
5,~6]{Cochran-Harvey-Horn:2012-1} generalize to their $\Z_{(p)}$-homology
analogues.  For those who are not familiar with these results, we
spell out the statements.

\begin{theorem}[Obstructions to being $\Z_{(p)}$-homology $0$-positive;
  c.f. {\cite[Section~4]{Cochran-Harvey-Horn:2012-1}}]
  \label{theorem:0-bipolar-obstructions}
  If a knot $K$ in a $\Z_{(p)}$-homology 3-sphere is
  $\Z_{(p)}$-homology $0$-positive, then the following hold:
  \begin{enumerate}
  \item The average function $\overline\sigma_K(\theta)$ of Cha-Ko's
    signature \cite{Cha-Ko:2000-1} is nonpositive.
  \item Ozsv\'ath-Szab\'o-Rasmussen $\tau$-invariant
    \cite{Ozsvath-Szabo:2003-1,Rasmussen:2003-01} of $K$ is
    nonnegative.
  \item The $(\pm 1)$-surgery manifold, say $N$, of $K$ bounds a
    4-manifold $W$ with positive definite intersection form on
    $H_2(W)/\text{torsion}$ and $H_1(W;\Z_p)=0$.  Consequently the
    Ozsv\'ath-Szab\'o $d$-invariant \cite{Ozsvath-Szabo:2003-2} of $N$
    associated to its unique $\spin^c$ structure is non-positive.
  \item In addition, if $p=2$ and $\overline\sigma_K(2\pi k/2^a)=0$
    for all $k$, then the Fintushel-Stern-Hedden-Kirk obstruction
    ~\cite{Hedden-Kirk:2010-1} associated to the $2^a$-fold cyclic
    branched cover vanishes.
  \end{enumerate}
  Consequently, if $K$ is $\Z_{(p)}$-homology $0$-bipolar, then
  $\overline\sigma_K(\theta)=0$, $\tau(K)=0$, and Hom's invariant
  $\epsilon(K)=0$ ~\cite{Hom:2011-1}.
\end{theorem}

We remark that in order to generalize the Levine-Tristram signature
result in \cite[Section~4]{Cochran-Harvey-Horn:2012-1}, we employ, in
Theorem~\ref{theorem:0-bipolar-obstructions}~(1), a generalization of
the Levine-Tristram signature to knots in rational homology spheres
which was introduced by Cha and Ko~\cite{Cha-Ko:2000-1}.  We give a
description of the invariant $\overline\sigma_K$ and a proof of
Theorem~\ref{theorem:0-bipolar-obstructions} (1) in
Appendix~\ref{section:signature-and-0-positivity}.

The proofs of other parts of
Theorem~\ref{theorem:0-bipolar-obstructions}, and
Theorem~\ref{theorem:d-invariant-obstruction} stated below, are
completely identical to those of the homotopy positive (resp.\ negative,
bipolar) cases given in \cite{Cochran-Harvey-Horn:2012-1}.

\begin{theorem}[Obstruction to being $\Z_{(p)}$-homology $1$-positive
  {\cite[Section~6]{Cochran-Harvey-Horn:2012-1}}]
  \label{theorem:d-invariant-obstruction}
  Suppose $K$ is a $\Z_{(p)}$-homology $1$-positive knot.  Then the
  $d$-invariant of $K$ associated to the $p^a$-fold cyclic branched
  cover $M$ is nonpositive, in the following sense: for some
  metabolizer $H\subset H_1(M)$ \textup{(}i.e. $|H|^2= |H_1(M)|$ and
  the $\Q/\Z$-valued linking form vanishes on $H$\textup{)} and a
  $\spin^c$ structure $s_0$ on $M$ corresponding to a spin structure
  on $M$, $d(M,s_0+\hat z)\le 0$ for any $z\in H$, where $\hat z$ is
  the Poincar\'e dual of~$z$.
\end{theorem}

\begin{remark}
  Note that Theorem \ref{theorem:d-invariant-obstruction} also follows
  from our Covering Positon
  Theorem~\ref{theorem:covering-positon-theorem} that is stated and
  proved in Section~\ref{section:covering_link_calculus}: if we take a
  $p^a$-fold branched cover $M$ of $S^3$ along $K$, then by
  Theorem~\ref{theorem:covering-positon-theorem}, $M$ bounds a
  $\Z_{(p)}$-homology $0$-position, to which a result of
  Ozsv\'ath-Szab\'o \cite{Ozsvath-Szabo:2003-2} applies to allow us to conclude
  that the $d$-invariant of $M$ with appropriate $\spin^c$ structures
  is non-positive.  Indeed our proof of Theorem
  \ref{theorem:covering-positon-theorem} involves some arguments which
  are similar to those
  in~\cite[Section~6]{Cochran-Harvey-Horn:2012-1}.
\end{remark}

We do not know whether the Rasmussen $s$-invariant obstruction in
\cite{Cochran-Harvey-Horn:2012-1} has a $\Z_{(p)}$-homology analogue.
Even the following weaker question is still left open.

\begin{question}
  If a knot $K$ is slice in a homology 4-ball (or more generally in a
  $\Z_{(p)}$- or $\Q$-homology 4-ball), then does $s(K)$ vanish?
\end{question}

We can also relate $\Z_{(p)}$-homology positivity to amenable von
Neumann $\rho$-invariants, following the idea of
\cite[Section~5]{Cochran-Harvey-Horn:2012-1} and using techniques
of~\cite{Cha:2010-01}.  We discuss this in more detail in
Appendix~\ref{section:rho-invariant-and-positivity} (see
Theorem~\ref{theorem:rho-invariant-and-positivity}).

\section{Covering link calculus}
\label{section:covering_link_calculus}

We follow \cite{Cha-Kim:2008-1} to give a formal description of
covering links.  Suppose $L$ is a link in a
$\Z_{(p)}$-homology 3-sphere~$Y$.  We consider the following two
operations that give new links: (C1) taking a sublink of $L$, and
(C2) taking the pre-image of $L$ in the $p^a$-fold cyclic cover of $Y$
branched along a component of $L$ with vanishing $\Q/\Z$-valued
self-linking.  Note that the $p^a$-fold cyclic branched cover
is again a $\Z_{(p)}$-homology sphere, by the argument of Casson and
Gordon~\cite[Lemma~2]{Casson-Gordon:1986-1}.

\begin{definition}
  \label{definition:covering-links}
  A link $\tilde L$ obtained from $L$ by a finite sequence of the
  operations (C1) and/or (C2) above is called a \emph{$p$-covering
    link of $L$ of height $\le h$}, where $h$ is the number of~(C2)
  operations.
\end{definition}

Often we say that the link $\tilde L$ in
Definition~\ref{definition:covering-links} is a \emph{height $h$}
$p$-covering link. This is an abuse of terminology without the $\leq$
sign; it is more appropriate to define the height of $\tilde L$ to be
minimum of the number of (C2) moves over all sequences of (C1) and
(C2) operations which produce $\tilde L$ from~$L$.  In all statements
in this paper, this abuse does not cause any problem, since if we
assign a height to a covering link which is not minimal, we obtain a
weaker conclusion from Theorem \ref{theorem:covering-positon-theorem}
than is optimal.

We remark that for an oriented link $L$, a covering link $\tilde L$ has
a well-defined induced orientation.  Since the choice of an
orientation is irrelevant for the purpose of the study of
$n$-positivity (or negativity, bipolarity) as discussed after
Theorem~\ref{theorem:basic-construction-and-positivity}, in this paper
we also call a covering link with some component's orientation
reversed a covering link.

\subsection{Covering Positon/Negaton Theorem}

The main theorem of this section is the following:

\begin{theorem}[Covering Positon/Negaton Theorem]
  \label{theorem:covering-positon-theorem}
  For $n>k$, a height $k$ $p$-covering link of a $\Z_{(p)}$-homology
  $n$-positive \textup{(}resp.\ negative, bipolar\textup{)} link is
  $\Z_{(p)}$-homology $(n-k)$-positive \textup{(}resp.\ negative,
  bipolar\textup{)}.
\end{theorem}

We remark that this may be compared with Covering Solution Theorem
\cite[Theorem~3.5]{Cha:2007-2} that provides a similar method for the
$n$-solvable filtration of~\cite{Cochran-Orr-Teichner:1999-1}.

\begin{proof}
  It suffices to prove the positivity case.  Suppose $L$ is a link in
  a $\Z_{(p)}$-homology 3-sphere~$Y$, and $V$ is a $\Z_{(p)}$-homology
  $n$-positon for $L$ with slicing disks~$\Delta_i$.  It suffices to
  show the following:
  \begin{enumerate}
  \item A sublink of $L$ is $\Z_{(p)}$-homology $n$-positive.
  \item Suppose $n>0$, $\tilde Y$ is a $p^a$-fold cyclic branched
    cover of $Y$ along the first component of $L$, and $\tilde L$ is
    the pre-image of $L$ in $\tilde Y$.  Then $\tilde L$ is
    $\Z_{(p)}$-homology $(n-1)$-positive.
  \end{enumerate}

  Item (1) has been already discussed in
  Theorem~\ref{theorem:basic-construction-and-positivity}.  To show
  (2), we first observe that by
  Lemma~\ref{lemma:basic-properties-of-Z_p-positon} (1) applied to the
  first component as a sublink (see also the observation below
  Lemma~\ref{lemma:basic-properties-of-Z_p-positon}),
  $H_1(V-\Delta_1;\Z_{p^a})$ is isomorphic to $\Z_{p^a}$ and generated
  by a meridional curve of~$\Delta_1$.  Therefore we can define the
  $p^a$-fold cyclic branched cover of $\tilde V$ of $V$ along
  $\Delta_1$.  In addition, since the inclusion induces an isomorphism
  $H_1(Y-\partial\Delta_1;\Z_{p^a})\to H_1(V-\Delta_1;\Z_{p^a})$, it
  follows that $\partial \tilde V$ is the ambient space $\tilde Y$ of
  the covering link~$L$, by the above discussion on the zero-framing
  and branched covering construction for~$Y$, namely Lemma
  \ref{lemma:positivity-zero-framing} and its preceding paragraph.

  We will show that $\tilde V$ is a $\Z_{(p)}$-homology
  $(n-1)$-positon for~$\tilde L$.

  For notational convenience, we denote the pre-image of a subset
  $A\subset V$ under $\tilde V \to V$ by $\tilde A \subset \tilde V$.
  Observe that components of the $\tilde\Delta_i$ form disjoint
  slicing disks for $\tilde L$ in~$\tilde V$, since each $\Delta_i$ is
  simply connected.  Let $\mu \subset V$ be a meridional circle
  for~$\Delta_1$.  Then $\tilde \mu$ is a meridional circle of the
  2-disk $\tilde\Delta_1$.  By
  Lemma~\ref{lemma:basic-properties-of-Z_p-positon} (1),
  $H_1(V-\Delta_1,\mu;\Z_p)=0$.  From the following fact, it follows
  that $H_1(\tilde V-\tilde\Delta_1,\tilde\mu;\Z_p) \cong 0$.

  \begin{lemma}[{\cite[Corollary~4.10]{Cochran-Harvey:2007-01},
    \cite[p.~910]{Cha:2007-2}}]
    \label{lemma:betti-number-bound-for-p-covers}
    Suppose $G$ is a $p$-group and $C_*$ is a projective chain complex
    over the group ring~$\Z_p G$ with $C_k$ finitely generated.  Then,
    for any $k$, $\dim_{\Z_p} H_k(C_*) \le |G|\cdot\dim_{\Z_p} H_k(\Z_p
    \otimes_{\Z_p G} C_*)$.
  \end{lemma}

  That $H_1(\tilde V-\tilde\Delta_1,\tilde\mu;\Z_p) \cong 0$ implies
  that $H_1(\tilde V-\tilde\Delta_1,\tilde\mu;\Z_{(p)}) \cong 0$ and
  $H_1(\tilde V-\tilde\Delta_1;\Z_{(p)})$ is generated by the class
  of~$\tilde\mu$.  It follows that $H_1(\tilde V;\Z_{(p)})=0$ since
  the homotopy type of $\tilde V$ is obtained by attaching a 2-cell to
  $\tilde V-\tilde \Delta_1$ along $\tilde\mu$.

  For later use, we also claim that $H_1(\tilde
  V-\tilde\Delta_1;\Z_{(p)})\cong\Z_{(p)}$, generated by the
  class~$[\tilde\mu]$.  For, since the map $H_1(\tilde
  V-\tilde\Delta_1;\Z_{(p)}) \to H_1(V-\Delta_1;\Z_{(p)})\cong
  \Z_{(p)}$ sends $[\tilde\mu]$ to $p^a [\mu]$, which is a multiple of
  a generator, $H_1(\tilde V-\tilde\Delta_1;\Z_{(p)})$ is not
  $\Z_{(p)}$-torsion.  The claim follows.  A consequence is
  that $H_1(\tilde V-\tilde\Delta_1;\Z_p)$ is isomorphic to $\Z_p$ and
  generated by $[\tilde\mu]$ (see the paragraph after
  Lemma~\ref{lemma:basic-properties-of-Z_p-positon}).

  Observe that the $p^a$-fold cover $\tilde V-\bigsqcup\tilde\Delta_i$
  of $V-\bigsqcup\Delta_i$ has fundamental group
  \[\textstyle
  \pi_1(\tilde V-\bigsqcup\tilde\Delta_i) = \Ker \{
  \pi_1(V-\bigsqcup\Delta_i) \to
  H_1(V-\bigsqcup\Delta_i)/\text{torsion}\cong \Z^m \to \Z_{p^a} \}
  \]
  where final map sends the first meridian to $1\in\Z_{p^a}$ and the other
  meridians to~$0$.  (For convenience we view the fundamental group of
  a covering space as a subgroup of the fundamental group of its base
  space.)  Also, from the definition we have
  \[\textstyle
  \cP^1 \pi_1(V-\bigsqcup \Delta_i) = \Ker\{\pi_1(V-\bigsqcup\Delta_i)
  \rightarrow H_1(V-\bigsqcup\Delta_i;\Z_{(p)})\},
  \]
  where the above map with kernel $\cP^1 \pi_1(V-\bigsqcup \Delta_i)$
  decomposes as
  \[\textstyle
  \pi_1(V-\bigsqcup\Delta_i) \to
  H_1(V-\bigsqcup\Delta_i)/\text{torsion}\cong \Z^m \hookrightarrow
  H_1(V-\bigsqcup\Delta_i;\Z_{(p)})
  \]
  with the rightmost map injective, by
  Lemma~\ref{lemma:basic-properties-of-Z_p-positon}~(\ref{lemma-case:H1-of-slice-disk-complement-in-positon})
  and the paragraph below
  Lemma~\ref{lemma:basic-properties-of-Z_p-positon}.  From this it
  follows that $\cP^1 \pi_1(V-\bigsqcup\Delta_i) \subset
  \pi_1(\tilde V-\bigsqcup\tilde\Delta_i)$.

  Suppose that $S_j$ are disjointly embedded surfaces in $V$
  satisfying Definition~\ref{definition:homology-n-positon}.  By
  definition and the above observation, we have
  \[\textstyle
  \pi_1(S_j) \subset \cP^n \pi_1(V-\bigsqcup \Delta_i) \subset \cP^1
  \pi_1(V-\bigsqcup \Delta_i) \subset \pi_1(\tilde
  V-\bigsqcup\tilde\Delta_i).
  \]
  By the lifting criterion, it follows that the surfaces $S_j$ lift to
  $\tilde V -\bigsqcup \tilde\Delta_i$, that is, $\tilde S_j$ consists
  of $p^a$ lifts $\{\Sigma_{j,k}\}_{k=1}^{p^a}$ of~$S_j$.  The
  $\Sigma_{j,k}$ are mutually disjoint, and are disjoint to the
  slicing disks for $\tilde L$.  Furthermore, each $\Sigma_{j,k}$ has
  self intersection $+1$ in $\tilde V$ since so does $S_j$ in~$V$.  As
  subgroups of $\pi_1(V-\bigsqcup \Delta_i)$, we have
  \[\textstyle
  \pi_1(\Sigma_{j,k}) = \text{a conjugate of } \pi_1(S_j)\subset \cP^n
  \pi_1(V-\bigsqcup \Delta_i) \subset \cP^{n-1}\pi_1(\tilde V-\bigsqcup
  \tilde\Delta_i).
  \]
  It is easy to see from the definitions that the last inclusion
  follows from the fact that $\cP^1 \pi_1(V-\bigsqcup \Delta_i)
  \subset \pi_1(\tilde V-\bigsqcup \tilde\Delta_i)$.

  It remains to prove that the $\Sigma_{j,k}$ generate $H_2(\tilde
  V)/\text{torsion}$.  Denote the $i$th Betti number by
  $b_i(-;\Q)=\dim_\Q H_i(-;\Q)$ and $b_i(-;\Z_p)=\dim_{\Z_p}
  H_i(-;\Z_p)$.  We claim the following:
  \begin{align*}
    b_2(\tilde V;\Q) &\le b_2(\tilde V;\Z_p) = b_2(\tilde V-\tilde
    \Delta_1;\Z_p) \le p^a\cdot b_2(V-\Delta_1;\Z_p)
    \\
    &= p^a \cdot b_2(V;\Z_p) = p^a\cdot b_2(V;\Q).
  \end{align*}
  The first inequality is from the universal coefficient theorem.  The
  second part is obtained by a Mayer-Vietoris argument for $(\tilde
  V-\tilde\Delta_1) \cup \text{(2-cell)} \simeq \tilde V$, using that
  our previous observation that $H_1(\tilde
  V-\tilde\Delta_1;\Z_p)\cong \Z_p$ is generated by~$\tilde\mu$.  The
  third part follows from
  Lemma~\ref{lemma:betti-number-bound-for-p-covers}.  The fourth part
  is again by a Mayer-Vietoris argument as before, for $(V-\Delta_1)
  \cup \text{(2-cell)} \simeq V$.  The last part is shown by
  Lemma~\ref{lemma:basic-properties-of-Z_p-positon}~(2) and its
  accompanying paragraph; the fact that the torsion part of $H_2(V)$
  has order coprime to~$p$.

  Now, let $h$ be the map from the free abelian group $F$ generated by
  the $\Sigma_{j,k}$ into $H_2(\tilde V)/\text{torsion}$ sending
  $\Sigma_{j,k}$ to its homology class.  The fact that the surfaces
  $\Sigma_{j,k}$ are disjoint and have self intersection one implies
  that the composition
  \[
  F\xrightarrow{h} H_2(\tilde V)/\text{torsion}
  \xrightarrow{\lambda^{ad}} \Hom(H_2(\tilde
  V)/\text{torsion},\Z) \xrightarrow{h^*} \Hom(F,\Z)
  \]
  is the adjoint of a form represented by the identity matrix, where
  $\lambda^{ad}$ is (the adjoint of) the intersection pairing.  It
  follows that the initial map $h$ is injective.  Since $b_2(\tilde
  V;\Q) \le p^a\cdot b_2(V;\Q) = \rank F$, it follows that $b_2(\tilde
  V;\Q)=p^a\cdot b_2(V;\Q)=\rank F$.  Since all the terms in the above
  sequence are free abelian groups of the same rank, it follows that
  all the maps, particularly $h$, are isomorphisms.  This completes
  the proof of Theorem \ref{theorem:covering-positon-theorem}.
\end{proof}

\begin{remark}
  Generalizing Definition~\ref{definition:homology-n-positon}, we can
  define a relation on links, similarly to
  \cite[Definition~2.1]{Cochran-Harvey-Horn:2012-1}: for links
  $L\subset Y$ and $L\subset Y'$ in $\Z_{(p)}$-homology spheres, we
  write $L \ge_{(\Z_{(p)},n)} L'$ if there is a 4-manifold $V$
  satisfying the following.  (1) $H_1(V;\Z_{(p)})=0$.  (2) There exist
  disjointly embedded annuli $C_i$ in $V$ satisfying
  $\partial(V,\bigsqcup C_i) = (Y,L)\sqcup -(Y',L')$ and that both $C_i
  \cap Y$, $C_i \cap Y'$ are nonempty for each $i$.  (3) The
  intersection pairing on $H_2(V)/\text{torsion}$ is positive definite
  and there are disjointly embedded surfaces $S_j$ in $V-\bigsqcup C_i$
  which generate $H_2(V)/\text{torsion}$ and satisfy
  $\pi_1(S_j)\subset \cP^n \pi_1(V-\bigsqcup C_i)$.  Note that $L$ is
  $\Z_{(p)}$-homology $n$-positive if and only if $L
  \ge_{(\Z_{(p)},n)}$ (unlink).

  Then the arguments of the proof of
  Theorem~\ref{theorem:covering-positon-theorem} show the following
  generalized statement: \emph{If $L \ge_{(\Z_{(p)},n)} L'$ and
    $\tilde L$ and $\tilde L'$ are height $k$ $p$-covering links of
    $L$ and $L'$ obtained by applying the same sequence of operations
    \textup{(C1)} and \textup{(C2)}, then $\tilde L \ge_{(\Z_{(p)},n-k)} \tilde L'$.}
  In order to interpret the idea of the same sequence of operations,
  we use the annuli $C_i$ to pair up components of $L$ and $L'$.
\end{remark}

\subsection{Examples: Bing doubles}
\label{subsection:bing-doubles}

Theorem~\ref{theorem:covering-positon-theorem} applied to
covering link calculus results for Bing doubles in the literature
immediately gives various interesting examples which are often
topologically slice links.

We denote by $B(L)$ the Bing double of a link $L$ in $S^3$, and for $n
\geq 1$ define $B_n(L) = B(B_{n-1}(L))$ to be the $n$th iterated Bing
double, where by convention, $B_0(L)$ is $L$ itself.  In this paper
Bing doubles are always untwisted.  We remark that if $K$ is
topologically (resp.\ smoothly) slice, then $B_n(K)$ is topologically
(resp.\ smoothly) slice.  The converse is a well-known open problem.

For an oriented knot $K$, we denote the reverse of~$K$ by $K^r$.

\begin{theorem}
  \label{theorem:covering-link-iterated-bing-double}
  If $B_n(K)$ is $\Z_{(p)}$-homology $(k+2n-1)$-positive
  \textup{(}resp.\ negative, bipolar\textup{)}, then $K\#K^r$ is
  $\Z_{(p)}$-homology $k$-positive \textup{(}resp.\ negative,
  bipolar\textup{)}.
\end{theorem}

\begin{proof}
  The following fact is due to
  Cha-Livingston-Ruberman~\cite{Cha-Livingston-Ruberman:2006-1},
  Cha-Kim~\cite{Cha-Kim:2008-1}, Livingston-Van
  Cott~\cite{Livingston-VanCott:2011-01} and Van
  Cott~\cite{VanCott:2009-01}: for any knot $K$ in $S^3$ and any prime
  $p$, $K\# K^r$ is a height $(2n-1)$ $p$-covering link of $B_n(K)$.
  Therefore by Theorem~\ref{theorem:covering-positon-theorem}, the
  conclusion follows.
\end{proof}

\begin{corollary}
  \label{corollary:iterated-bing-double-for-nonzero-tau}
  If $\tau(K) \ne 0$, then $B_n(K)$ is not $\Z_{(p)}$-homology
  $(2n-1)$-bipolar for any~$p$.
\end{corollary}

It is well-known that there are topologically slice knots $K$ with
$\tau(K)\ne 0$ (e.g.\ $K=$ the positive Whitehead double of any knot
$J$ with $\tau(J)>0$, by Hedden~\cite{Hedden:2007-1}).  For such a
knot $K$, Corollary~\ref{corollary:iterated-bing-double-for-nonzero-tau}
implies that $B_n(K)$ is a topologically slice link which is not
$(2n-1)$-bipolar.

\begin{proof}[Proof of Corollary \ref{corollary:iterated-bing-double-for-nonzero-tau}]
  Suppose $B_n(K)$ is $\Z_{(p)}$-homology $(2n-1)$-positive for some
  $p$.  Then $K\# K^r$ is $\Z_{(p)}$-homology $0$-positive by
  Theorem~\ref{theorem:covering-link-iterated-bing-double}.  It
  follows that $2\tau(K) = \tau(K\# K^r)\ge 0$ by
  Theorem~\ref{theorem:0-bipolar-obstructions}.  Similarly,
  $\tau(K)\le 0$ if $B_n(K)$ is $\Z_{(p)}$-homology $(2n-1)$-negative.
\end{proof}

While the non-bipolarity of the above links is easily derived by
applying our method, we do not know the precise bipolar height of
these links.  The following lemma is useful in producing highly
bipolar links:

\begin{lemma}
  \label{lemma:positivity-or-bing-double}
  If $L$ is $n$-positive \textup{(}resp.\ negative, bipolar\textup{)},
  then $B_k(L)$ is $(k+n)$-positive \textup{(}resp.\ negative,
  bipolar\textup{)}.  The $\Z_{(p)}$-homology analogue holds too.
\end{lemma}

\begin{proof}
  We may assume $k=1$ by induction.  Let $V$ be an $n$-positon for
  $L=K_1\sqcup\cdots\sqcup K_m$ with slicing disks~$\Delta_i$.  We
  proceed similarly to the proof of
  Theorem~\ref{theorem:basic-construction-and-positivity}~(\ref{theorem-case:positivity-generalize-satellite}),
  except that we need a stronger conclusion.  We identify a tubular
  neighborhood $\nu(\Delta_i)$ with $\Delta_i\times D^2$ as usual.
  Note that $B(L)$ is obtained by replacing, for all $i$,
  $(\nu(K_i),K_i)$ with the standard Bing link $L_0$, which is a
  2-component link in $S^1\times D^2$.  Viewing $L_0$ as a link in
  $S^3$ via the standard embedding $S^1\times D^2 \subset \partial
  (D^2\times D^2)\cong S^3$, $L_0$ is a trivial link.  Consequently
  $L_0$ bounds disjoint slicing disks, say $D_{i,1}$ and $D_{i,2}$, in
  $D^2\times D^2\cong D^4$.  Replacing $(\nu(\Delta_i),\Delta_i)$ with
  $(D^2\times D^2, D_1\sqcup D_2)$ for each $i$, we see that $B(L)$ has
  slicing disks, say $D_\ell$ in~$V$.  Since $\pi_1(V)=0$,
  $\pi_1(V-\bigsqcup\nu(\Delta_i))$ is normally generated by the
  meridians $\mu_i$ of the~$K_i$.  Since the meridional curve $*\times
  S^1 \subset S^1\times D^2$ is homotopic to a commutator of meridians
  of the components of $L_0$ in $S^1 \times D^2-L_0$, it follows that
  the image of $\pi_1(V-\bigsqcup\nu(\Delta_i))$ lies in
  $\pi_1(V-\bigsqcup D_{\ell})^{(1)}$.  For the surfaces $S_j$ in
  Definition~\ref{definition:homotopy-n-positon}, since $\pi_1(S_j)
  \subset \pi_1(V-\bigsqcup\nu(\Delta_i))^{(n)}$, it follows that
  $\pi_1(S_j) \subset \pi_1(V-\bigsqcup D_\ell)^{(n+1)}$.  This shows that
  $V$ is an $(n+1)$-positon for~$B(L)$.

  For the $\Z_{(p)}$-homology analogue, we proceed similarly.  In this
  case, the argument using that $\pi_1(V-\bigsqcup\nu(\Delta_i))$ is
  normally generated by the meridians does not work.  Instead, we
  appeal to the following: first we claim that the image of
  $\pi_1(V-\bigsqcup\nu(\Delta_i))$ in $\pi_1(V-\bigsqcup D_\ell)$
  lies in the subgroup $\cP^1\pi_1(V-\bigsqcup D_\ell)$.  For, we have
  a commutative diagram
  \[
  \begin{diagram}
    \node{\textstyle\pi_1(V-\bigsqcup\nu(\Delta_i))} \arrow{e}\arrow{s}
    \node{\textstyle H_1(V-\bigsqcup\nu(\Delta_i);\Z_{(p)}) \cong
      (\Z_{(p)})^m} \arrow{s}
    \\
    \node{\textstyle\pi_1(V-\bigsqcup D_\ell)} \arrow{e,t}{\psi}
    \node{\textstyle H_1(V-\bigsqcup D_\ell;\Z_{(p)})}
  \end{diagram}
  \]
  with vertical maps induced by the inclusion.  Here the isomorphism
  in the upper right corner is obtained by
  Lemma~\ref{lemma:basic-properties-of-Z_p-positon}~(\ref{lemma-case:H1-of-slice-disk-complement-in-positon}).
  Since $(\Z_{(p)})^m$ is generated by meridians of $L$ which are
  null-homologous in the exterior of $B(L)$, the right vertical arrow
  is a zero map.  Since $\cP^{1}\pi_1(V-\bigsqcup D_\ell)$ is the
  kernel of $\psi$, the claim follows.  Now, from the claim we obtain
  \[\textstyle
  \pi_1(S_j)\subset \cP^n \pi_1(V-\bigsqcup \nu(\Delta_i)) \subset
  \cP^{n+1} \pi_1(V-\bigsqcup D_\ell)
  \]
  as desired.
\end{proof}

\begin{theorem}
  \label{theorem:iterated-bing-doubles-of-CH-knot}
  Suppose $K$ is a knot in $S^3$ which is $0$-bipolar and $1$-positive
  but not $\Z_{(p)}$-homology $1$-bipolar for some $p$.  Then $B_n(K)$
  is $n$-bipolar but not
  \textup{(}$\Z_{(p)}$-homology\textup{)} $2n$-bipolar.
\end{theorem}

For example, we obtain topologically slice links which are $n$-bipolar
but not $2n$-bipolar, by applying
Theorem~\ref{theorem:iterated-bing-doubles-of-CH-knot} to the knots
presented by Cochran and Horn~\cite{Cochran-Horn:2012-1}; they
constructed topologically slice knots $K$ which are $0$-bipolar,
$1$-positive but not $1$-bipolar.  Indeed $K$ is shown not to be
$1$-negative by exhibiting that a certain $d$-invariant of the double
branched cover associated to a metabolizer element is negative in the
sense of \cite[Theorem~6.2]{Cochran-Harvey-Horn:2012-1} and
Theorem~\ref{theorem:d-invariant-obstruction}.  It follows that their
$K$ is not $\Z_{(2)}$-homology $1$-negative by
Theorem~\ref{theorem:d-invariant-obstruction}.

\begin{proof}[Proof of Theorem~\ref{theorem:iterated-bing-doubles-of-CH-knot}]
  By Lemma~\ref{lemma:positivity-or-bing-double}, $B_n(K)$ is
  $n$-bipolar since $K$ is $0$-bipolar.  If $B_n(K)$ is $2n$-bipolar
  then it is $\Z_{(p)}$-homology $2n$-bipolar for all primes $p$.  By
  Theorem~\ref{theorem:covering-link-iterated-bing-double} for $k=1$
  we see that $K\#K^r$ is $\Z_{(p)}$-homology $1$-bipolar.  Since $K$
  is $\Z_{(p)}$-homology $1$-positive, the concordance inverse $-K^r$
  of $K^r$ is $\Z_{(p)}$-homology $1$-negative by Theorem
  \ref{theorem:basic-construction-and-positivity}~(\ref{theorem-case:positivity-mirror-image}),~(\ref{theorem-case:positivity-changing-string-orientation}).
  We have that $K$ is concordant to $K \# K^r \# -K^r$.  By Theorem
  \ref{theorem:basic-construction-and-positivity}~(\ref{theorem-case:positivity-split-union}),~(\ref{theorem-case:positivity-band-sum}),
  this latter knot is $\Z_{(p)}$-homology $1$-negative since both $K\#
  K^r$ and $-K^r$ are $1$-negative.  Therefore so is $K$ by Theorem
  \ref{theorem:basic-construction-and-positivity}~(\ref{theorem-case:positivity-slice-link}).
  This contradicts the hypothesis.
\end{proof}

The examples obtained in
Theorem~\ref{theorem:iterated-bing-doubles-of-CH-knot} illustrate the
non-triviality of the $n$-bipolar filtration for links for
higher~$n$.  Note that there are still some limitations: the number of
components of our link $L=B_n(K)$ grows exponentially on $n$, and the
height $\BH(L) := \max\{h\mid L$ is $h$-bipolar$\}$ is not precisely
determined; we only know that $n \le \BH(B_n(K)) <2n$.  In the next
section we will give examples resolving these limitations.

We finish this section with a discussion of some covering link
calculus examples due to A. Levine~\cite{LevineA:2009-01}.  He
considered iterated Bing doubling operations associated to a binary
tree~$T$.  Namely, for a knot $K$, $B_T(K)$ is a link with components
indexed by the leaf nodes of $T$, which is defined inductively: for a
single node tree $T$, $B_T(K)=K$.  If $T'$ is obtained by attaching
two child nodes to a node $v$ of $T$, then $B_{T'}(K)$ is obtained
from $B_{T}(K)$ by Bing doubling the component associated to~$v$.  For
a link $L$, we denote by $\Wh_+(L)$ the link obtained by replacing
each component with its positive untwisted Whitehead double.  We
define the \emph{order} of a binary tree $T$ to be number of non-leaf
vertices i.e. one more than the number of trivalent vertices (since
the root vertex is bivalent), or equivalently, the number of leaf
vertices minus one.  We denote the order of $T$ by~$o(T)$.

\begin{theorem}
\label{theorem:mixed-whitehead-bing-double}
\begin{enumerate}
\item If $\tau(K)\ne 0$, then $\Wh_+(B_T(K))$ is not $o(T)$-bipolar.
\item For the Hopf link $H$, $\Wh_+(B_{T_1,T_2}(H))$ is not $(o(T_1) +
  o(T_2) + 1)$-bipolar.
\end{enumerate}
\end{theorem}

\begin{proof}
  Levine showed that $\Wh_+(B_T(K))$ has a knot $J$ with $\tau(J)>0$
  as a covering link of height $o(T)$.  Similarly for
  $\Wh_+(B_{T_1,T_2}(H))$, where the relevant covering link has height
  $o(T_1) + o(T_2) + 1$.
\end{proof}

\section{Raising the bipolar height by one}
\label{section:main-examples}

The goal of this section is to prove the following:

\begin{theorem}
  \label{theorem:main-example}
  For any $m>1$ and $n\ge 0$, there exist topologically slice
  $m$-component links in $S^3$ which are $n$-bipolar but not
  $(n+1)$-bipolar.
\end{theorem}

To construct links satisfying Theorem~\ref{theorem:main-example}, we will
introduce an operation that pushes a link into a deeper level of
the bipolar filtration.  To describe this behaviour of our operation, we
will use the following terminology:

\begin{definition}
  \label{definition:bipolar-height}
  The \emph{bipolar height} of a link $L$ is defined by
  \[
  \BH(L) := \max\{n \mid \text{$L$ is $n$-bipolar}\}.
  \]
  The \emph{$\Z_{(p)}$-bipolar height} is defined by
  \[
  \BH^p(L) := \max\{n \mid \text{$L$ is $\Z_{(p)}$-homology
    $n$-bipolar}\}.
  \]
  By convention, $\BH(L)=-1$ if $L$ is not $0$-bipolar.  Similarly
  for~c$\BH^p$.
\end{definition}

The proof of the following proposition is immediate from the
definitions, and is left to the reader to check.

\begin{proposition}\label{Proposition:bipolar-height-inequality}
  For any $L$ and for any $p$, we have $\BH(L) \le \BH^p(L)$.
\end{proposition}

The following refined notion will be the precise setting for our height raising theorem.

\begin{definition}
  \label{definition:bipolar-height-plus}
  We say that $L$ has property $\BH^p_+(n)$ if $\BH^p(L)=n$ and $L$ is
  $\Z_{(p)}$-homology $(n+1)$-positive and we say that $L$ has
  property $\BH^p_-(n)$ if $\BH^p(L)=n$ and $L$ is $\Z_{(p)}$-homology
  $(n+1)$-negative.
\end{definition}

Our operation is best described in terms of string links.  We always
draw a string link horizontally; components of a string link are
oriented from left to right, and ordered from bottom to top.  We denote
the closure of a string link $\beta$ by $\widehat\beta$.

From now on, we assume knots are oriented, so that a knot can be
viewed as a 1-component string link and vice versa.  (Recall that a
1-component string link is determined by its closure.)

\begin{definition}
  \label{definition:C(-)}
  \leavevmode\Nopagebreak
  \begin{enumerate}
  \item For a knot or a 1-string link $K$, we define $C(K)$ to be the
    2-component string link illustrated in
    Figure~\ref{figure:knot-doubling}.  The two parallel strands
    passing through $K$ are untwisted.  Note that the closure
    $\widehat{C(K)}$ is the Bing double of~$K$.

    \begin{figure}[H]
      \labellist
      \hair 0mm
      \pinlabel {$C(K)=$} [l] at 0 40
      \pinlabel {$K$} at 103 41
      \endlabellist
      \includegraphics{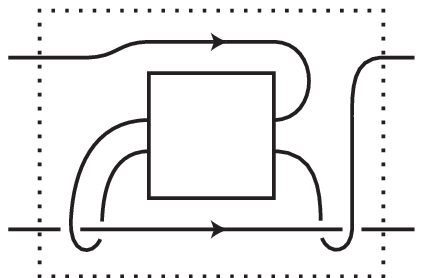}
      \caption{Bing doubling a knot $K$ as a 2-string link.}
      \label{figure:knot-doubling}
    \end{figure}

  \item For a 2-component string link $\beta$, we define $C(\beta)$ to
    be the 2-string link shown in Figure~\ref{figure:link-doubling}.
    (For now ignore the dashed arcs.)  As before, we take parallel
    strands passing through each of the strings of $\beta$ in an
    untwisted fashion.

    \begin{figure}[H]
      \labellist
      \hair 0mm
      \pinlabel {$C(\beta)=$} [l] at 0 58
      \pinlabel {$\beta$} at 97 65
      \endlabellist
      \includegraphics{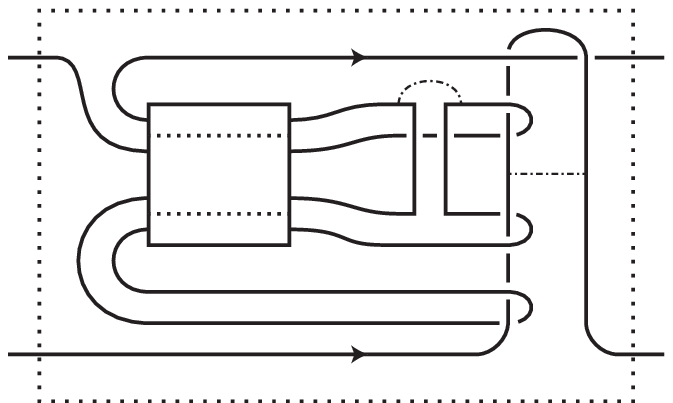}
      \caption{A doubling operation for 2-string links.}
      \label{figure:link-doubling}
    \end{figure}
  \end{enumerate}
\end{definition}

The key property of the operation $C(-)$ is the following height
raising result:

\begin{theorem}
  \label{theorem:bipolar-height-raising}
  Suppose $\beta$ is a string link with one or two components
  such that $\widehat\beta$ has property $\BH^p_+(n)$.  Then the link
  $\widehat{C(\beta)}$ has property $\BH^p_+(n+1)$.
\end{theorem}

We remark that the $\BH^p_-$ analogue of
Theorem~\ref{theorem:bipolar-height-raising} holds by taking mirror
images.

As the first step of the proof of
Theorem~\ref{theorem:bipolar-height-raising}, we make a few useful
observations as a lemma.

\begin{lemma}
  \label{lemma:bipolarity-of-C-doubling}
  If $\widehat\beta$ is $k$-positive \textup{(}resp.\ negative,
  bipolar\textup{)}, then $\widehat{C(\beta)}$ is $(k+1)$-positive
  \textup{(}resp.\ negative, bipolar\textup{)}.  If $\widehat\beta$ is
  $\Z_{(p)}$-homology $k$-positive \textup{(}resp.\ negative,
  bipolar\textup{)}, then $\widehat{C(\beta)}$ is $\Z_{(p)}$-homology
  $(k+1)$-positive \textup{(}resp.\ negative, bipolar\textup{)}.  If
  $\widehat\beta$ is topologically \textup{(}resp.\ smoothly\textup{)}
  slice, then so is~$\widehat{C(\beta)}$.
\end{lemma}

\begin{proof}
  For the case that $\beta$ is a 1-component string link, this is
  Lemma \ref{lemma:positivity-or-bing-double}, plus the observation
  that the Bing double of a slice link is slice, which follows from
  the argument in the proof of
  Theorem~\ref{theorem:basic-construction-and-positivity}~(\ref{theorem-case:positivity-generalize-satellite}).

  For the 2-component string link case, observe that
  $\widehat{C(\beta)}$ is obtained from the Bing double
  $B(\widehat\beta)$ by band sum of components: two pairs of
  components are joined.  The dashed arcs in
  Figure~\ref{figure:link-doubling} indicate where to cut to reverse
  these band sums.  From this the conclusion follows by
  Lemma~\ref{lemma:positivity-or-bing-double} and
  Theorem~\ref{theorem:basic-construction-and-positivity}~(\ref{theorem-case:positivity-band-sum}).
  The sliceness claim is shown similarly: the Bing double of a slice
  link is slice, as is the internal band sum of a slice link.
\end{proof}

Due to Lemma~\ref{lemma:bipolarity-of-C-doubling}, in order to prove Theorem~\ref{theorem:bipolar-height-raising} it remains to
show, given $\beta$ for which $\widehat{\beta}$ has property $\BH_+^p(n)$, that
$\widehat{C(\beta)}$ is not $\Z_{(p)}$-homology $(n+2)$-negative.  For this purpose we
will use covering link calculus.

In what follows we use the following notation.  For two knots $J_1$
and $J_2$, let $\ell(J_1, J_2)$ be the 2-component string link which
is the split union of string link representations of $J_1$ and $J_2$,
as the first and second components respectively.  For a 2-component
string link $\beta$, we denote by $\widetilde\beta$ the plat closure
of $\beta$; see Figure~\ref{figure:plat-closure}, which also indicates
the orientation which we give $\widetilde{\beta}$.  We say that
\emph{$\beta$ has unknotted components} if each strand of $\beta$ is
unknotted.

\begin{figure}[H]
  \labellist
  \hair 0mm
  \pinlabel {$\widetilde\beta=$} [l] at 0 33
  \pinlabel {$\beta$} at 80 27
  \pinlabel {$=$}  at 151 32
  \pinlabel {$\beta$} at 194 31
  \endlabellist
  \includegraphics{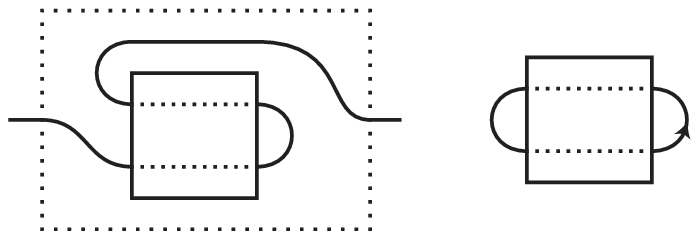}
  \caption{The plat closure of a 2-string links.}
  \label{figure:plat-closure}
\end{figure}

\def\sphat{\widehat{\hphantom{x}}}
\def\sphatt{\widehat{\hphantom{x}\vphantom{\sum}}}
\def\sptilde{\widetilde{\hphantom{x}}}

In what follows, $\cdot$ denotes the product of string links, given by
concatenation.

\begin{theorem}
  \label{theorem:height-one-covering-for-main-example}
  For any 2-component string link $\beta$ with unknotted components,
  the link $\widehat{C(\beta)}$ has, as a $p$-covering link of height
  one, the link
  \hbox{$\big(\beta\cdot\ell(\widetilde\beta{}^r,\widetilde\beta)\big)\sphat$}
  in $S^3$.
\end{theorem}

See Figure~\ref{figure:closure-doubled-link} for the initial link
$\widehat{C(\beta)}$, and
Figure~\ref{figure:covering-doubled-link} for its height one covering
link $\big(\beta\cdot\ell(\widetilde\beta{}^r,\widetilde\beta)\big)\sphat$.

\begin{figure}[H]
  \labellist
  \hair 0mm
  \pinlabel {$\beta$} at 40 60
  \endlabellist
  \[
  \widehat{C(\beta)} = \quad
  \vcenter{\hbox{\includegraphics{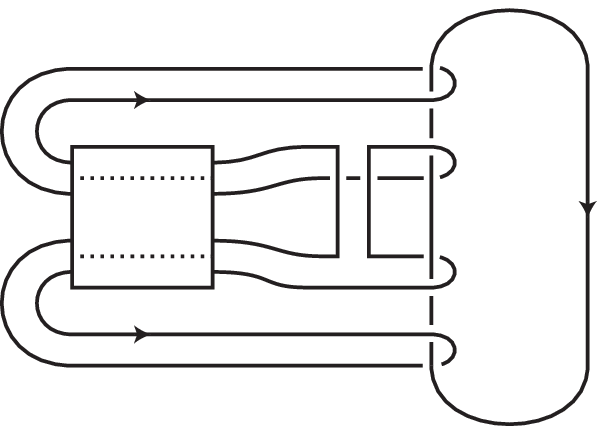}}}
  \]
  \caption{The closure of the 2-string link~$C(\beta)$.}
  \label{figure:closure-doubled-link}
\end{figure}

\begin{figure}[H]
  \labellist
  \hair 0mm
  \pinlabel {$\beta$} at 46 67
  \pinlabel {$\beta$} at 102 45
  \pinlabel {$\beta$} at 102 90
  \pinlabel {$=$} at 185 67
  \pinlabel {$\beta$} at 235 67
  \pinlabel {$\beta$} at 294 45
  \pinlabel {$\beta$} at 294 90
  \endlabellist
  \[
  \vcenter{\hbox{\includegraphics{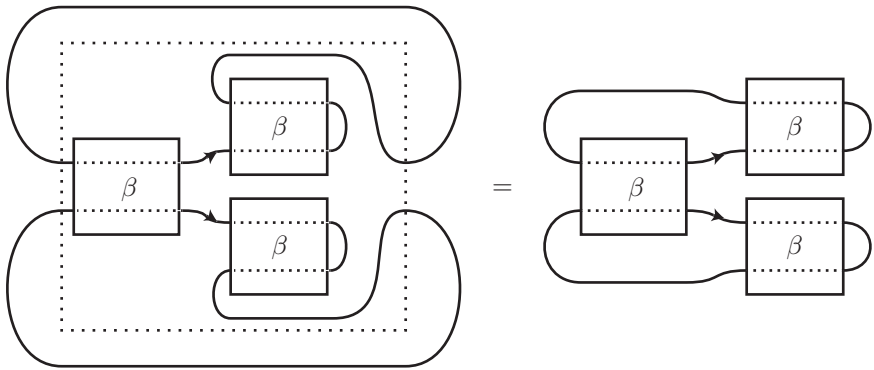}}}
  \]
  \caption{$\widehat{\beta\cdot\ell(\widetilde\beta{}^r,\widetilde\beta)}$
    is a height one covering link of $\widehat{C(\beta)}$.}
  \label{figure:covering-doubled-link}
\end{figure}

\begin{proof}
  Our proof mostly consists of pictures.
  While the statement of
  Theorem~\ref{theorem:height-one-covering-for-main-example} is
  independent of the choice of orientation of links due to our
  convention (see the remark after
  Definition~\ref{definition:covering-links}), we will work in the
  proof with oriented knots and (string) links in order to show
  clearly how the orientations of involved blocks match.  Indeed we
  will show that as oriented links the link in
  Figure~\ref{figure:covering-doubled-link} with the first (bottom)
  component's orientation reversed is a covering link of the link in
  Figure~\ref{figure:closure-doubled-link}.

  Choose $a$ such that $p^a \ge 5$, and take the $p^a$-fold cyclic
  cover branched along the right hand component of the link
  $\widehat{C(\beta)}$ illustrated in
  Figure~\ref{figure:closure-doubled-link}.  The covering space is
  obtained by cutting the 3-sphere along a disk whose boundary is the
  right hand component, and glueing $p^a$ copies of the result.  The
  covering link (with the pre-image of the branching component
  forgotten) in the branched cover is shown in
  Figure~\ref{figure:cyclic-branched-covering}.  Note that the ambient
  space is again $S^3$ since the branching component is an unknot.

  \begin{figure}[H]
    \labellist
    \hair 0mm
    \pinlabel {$\beta$} at 40 268
    \pinlabel {$\beta$} at 148 160
    \pinlabel {$\beta$} at 256 52
    \pinlabel {$*$} at 246 213
    \pinlabel {$*$} at 263 133
    \endlabellist
    \includegraphics[scale=.9]{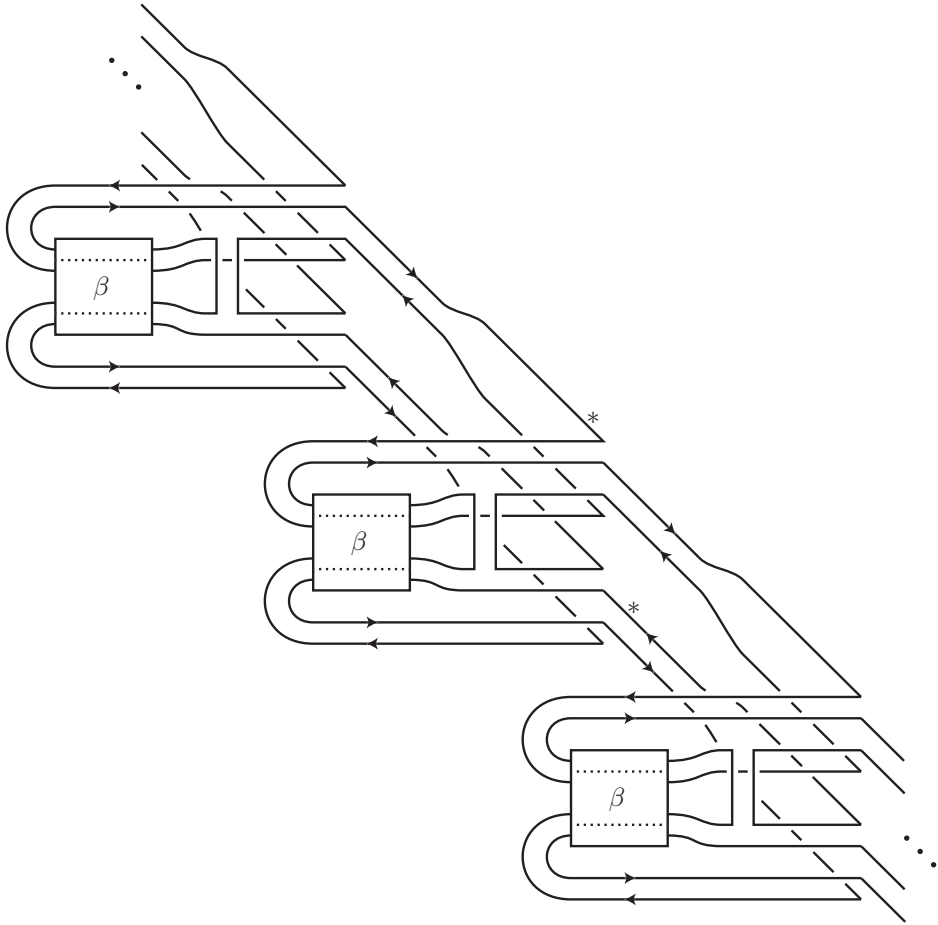}
    \caption{The cyclic branched covering.}
    \label{figure:cyclic-branched-covering}
  \end{figure}

  Forgetting all components except two marked by $*$ in
  Figure~\ref{figure:cyclic-branched-covering}, we obtain the 2-component
  link illustrated in Figure~\ref{figure:two-component-covering-link},
  since $p^a \ge 5$.  Since $\beta$ has unknotted components, the link
  in Figure~\ref{figure:two-component-covering-link} with the bottom
  component's orientation reversed is isotopic to the link
  $\big(\beta\cdot\ell(\widetilde\beta{}^r,\widetilde\beta)\big)\sphat$
  which is illustrated in Figure~\ref{figure:covering-doubled-link}.
  This completes the proof.
\end{proof}

  \begin{figure}[t]
    \labellist
    \hair 0mm
    \pinlabel {$\beta$} at 20 294
    \pinlabel {$\beta$} at 90 230
    \pinlabel {$\beta$} at 160 154
    \pinlabel {$\beta$} at 230 89
    \pinlabel {$\beta$} at 300 13
    \endlabellist
    \includegraphics{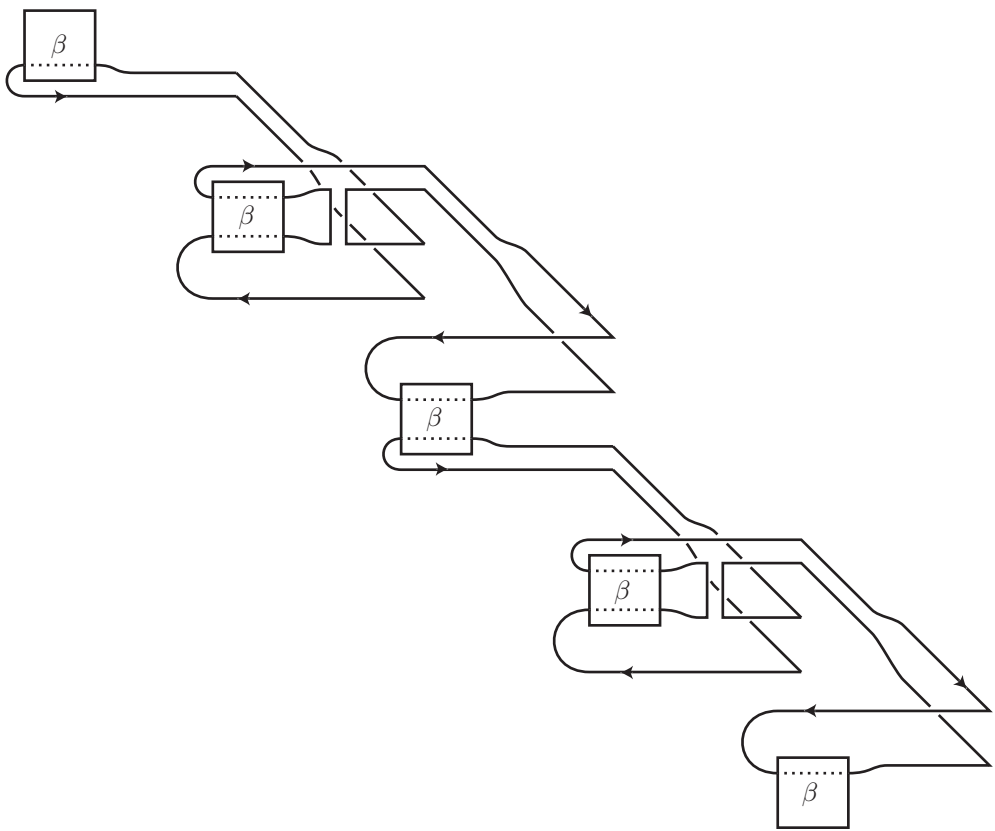}
    \caption{A two-component sublink of Figure~\ref{figure:cyclic-branched-covering}.}
    \label{figure:two-component-covering-link}
  \end{figure}

  Now we are ready to give a proof of
  Theorem~\ref{theorem:bipolar-height-raising}.

\begin{proof}[Proof of Theorem~\ref{theorem:bipolar-height-raising}]
  Suppose $\beta$  is such that $\widehat{\beta}$ has property $\BH^p_+(n)$.  We will show that $\widehat{C(\beta)}$ has property
  $\BH^p_+(n+1)$.  By Lemma~\ref{lemma:bipolarity-of-C-doubling}, it suffices to show that
  $\widehat{C(\beta)}$ is not $\Z_{(p)}$-homology $(n+2)$-bipolar.

  When $\beta$ is a $1$-component string link, $\widehat{C(\beta)} =
  B_1(\widehat{\beta})$; that is, the operator $C$ corresponds to Bing
  doubling.  So, if $n=0$, then
  Theorem~\ref{theorem:iterated-bing-doubles-of-CH-knot} (with $n=1$
  in the notation of that theorem) says that $\widehat{C(\beta)}$ is
  not $\Z_{(p)}$-homology $2$-bipolar.  For arbitrary $n$, observe
  that the proof of
  Theorem~\ref{theorem:iterated-bing-doubles-of-CH-knot} also works if
  we shift all bipolarity heights by a constant.

  Now suppose that $\beta$ is a 2-component string link.  We will show
  that $\widehat{C(\beta)}$ is not $\Z_{(p)}$-homology
  $(n+2)$-negative.  So let us suppose, for a contradiction, that
  $\widehat{C(\beta)}$ is $\Z_{(p)}$-homology $(n+2)$-negative.

  By Theorem~\ref{theorem:height-one-covering-for-main-example} and
  Theorem~\ref{theorem:covering-positon-theorem}, $(\beta\cdot
  \ell(\tilde\beta^r,\tilde\beta))\sphat$ is $\Z_{(p)}$-homology
  $(n+1)$-negative.  Since $\widehat\beta$ is $\Z_{(p)}$-homology
  $(n+1)$-positive and $\tilde\beta$ and $\tilde\beta^r$ are obtained
  from $\widehat\beta$ and $\widehat\beta^r$ by band sum of
  components, $\tilde\beta$ and $\tilde\beta^r$ are
  $\Z_{(p)}$-homology $(n+1)$-positive by
  Theorem~\ref{theorem:basic-construction-and-positivity}~(\ref{theorem-case:positivity-band-sum}).
  It follows that $(\ell(\tilde\beta^r,\tilde\beta)^{-1})\sphat$ is
  $\Z_{(p)}$-homology $(n+1)$-negative by
  Theorem~\ref{theorem:basic-construction-and-positivity}~(\ref{theorem-case:positivity-mirror-image}),~(\ref{theorem-case:positivity-changing-string-orientation})
  and~(\ref{theorem-case:positivity-split-union}), where for a string
  link $\gamma$, we denote the concordance inverse obtained by mirror
  image and reversing string orientation by $\gamma^{-1}$.  Then
  $\widehat\beta$, which is concordant to $(\beta\cdot
  \ell(\tilde\beta^r,\tilde\beta) \cdot
  \ell(\tilde\beta^r,\tilde\beta)^{-1})\sphat$, is $\Z_{(p)}$-homology
  $(n+1)$-negative by
  Theorem~\ref{theorem:basic-construction-and-positivity}~(\ref{theorem-case:positivity-band-sum})
  and~(\ref{theorem-case:positivity-slice-link}).  This contradicts
  the hypothesis that $\widehat\beta$ has property $\BH^p_+(n)$.  Therefore
  $\widehat{C(\beta)}$ is not $\Z_{(p)}$-homology $(n+2)$-negative, as
  desired.
\end{proof}

We are now ready to prove the promised result
(Theorem~\ref{theorem:main-example}) by giving explicit examples.  For
a knot $K$, define a sequence $\{C_n(K)\}$ inductively by $C_0(K) :=
K$, $C_{n+1}(K) := C(C_{n}(K))$ for $n >0$.  Note that $C_n(K)$ is a
2-component string link for $n>0$.

\begin{corollary}
  \label{corollary:main-example-detailed}
  Suppose $K$ is a knot which is topologically slice and $0$-bipolar
  and $\Z_{(p)}$-homology $1$-positive but not $\Z_{(p)}$-homology
  1-bipolar for some~$p$.  Then the 2-component link
  $\widehat{C_n(K)}$ is topologically slice and $n$-bipolar but not
  $(n+1)$-bipolar.
\end{corollary}

\begin{proof}
  Since $K$ is topologically slice and $0$-bipolar, $\widehat{C_n(K)}$
  is topologically slice and $n$-bipolar by applying
  Lemma~\ref{lemma:bipolarity-of-C-doubling} inductively.  So
  $\BH(\widehat{C_n(K)}) \ge n$.  Also, since $K$ has property $\BH^p_+(0)$,
  $\widehat{C_n(K)}$ has property $\BH^p_+(n)$ by applying
  Theorem~\ref{theorem:bipolar-height-raising} inductively.  In
  particular $\BH^p(\widehat{C_n(K)}) = n$, so $\BH(\widehat{C_n(K)})
  \le n$ by Proposition~\ref{Proposition:bipolar-height-inequality}.
  Thus $\BH(\widehat{C_n(K)}) = n$ and $\widehat{C_n(K)}$ is
  $n$-bipolar but not $(n+1)$-bipolar.
\end{proof}

Using one of the knots of Cochran-Horn~\cite{Cochran-Horn:2012-1} for
$K$, as discussed in the paragraph after
Theorem~\ref{theorem:iterated-bing-doubles-of-CH-knot}, the
2-component case of Theorem~\ref{theorem:main-example} follows from
Corollary~\ref{corollary:main-example-detailed} with $p=2$.  For the general
$m$-component case, the split union of the 2-component example and an
$(m-2)$-component unlink is a link with all the desired properties.

\section{String links and a subgroup of infinite rank}
\label{Section:string-links}

Thus far, the bipolar filtration of links with $m \ge 2$ components
has been a filtration by subsets; the set of links does not have a
well defined notion of connected sums.  In this section we consider
\emph{string links} to impose more structure.

\subsection{The bipolar filtration of string links}

Although we have already used some standard string link terminology in
Section~\ref{section:main-examples}, we begin by recalling the definitions
of string links and the concordance group of string links, which is
our object of study in this section.  Readers who are familiar with
string links may skip the following three paragraphs.

Fix $m$ distinct interior points in~$D^2$ and identify these with
$[m]:= \{1,\ldots,m\}$.  An $m$-component string link $\beta$ is a
collection of $m$ properly embedded oriented disjoint arcs in
$D^2\times I$ joining $(i,0)$ to $(i,1)$, $i\in [m]$.  Let $\beta_0$
and $\beta_1$ be two $m$-component string links.  The \emph{product}
$\beta_0\cdot\beta_1$ is defined by stacking cylinders.  We say that
$\beta_0$ and $\beta_1$ are \emph{concordant} if there are $m$
properly embedded disjoint disks in $(D^2\times I)\times I$ bounded by
$(\beta_0 \times 0) \cup ([m]\times
\partial I \times I) \cup (-\beta_1\times 1)$.  Concordance classes of
string links form a group under the product operation.  The identity
is the \emph{trivial string link} $[m]\times I \subset D^2\times I$.
The \emph{inverse} $\beta^{-1}$ of $\beta$ is defined to be its image
under the automorphism $(x,t) \to (x,1-t)$ on $D^2\times I$.  A string
link is \emph{slice} if it is concordant to the trivial string link.

The quotient space of $D^2\times I$ obtained by identifying $D^2\times
0$ and $D^2\times 1$ under the identity map and collapsing $x\times I$
to a point for each $x\in \partial D^2$ is diffeomorphic to~$S^3$.  The
\emph{closure} $\widehat\beta\subset S^3$ of $\beta$ is defined to be
the image of $\beta$ under the quotient map.  A string link $\beta$ is
slice if and only if $\widehat\beta$ is slice as a link.  Consequently
two string links $\beta_0$ and $\beta_1$ are concordant if and only if the
closure of $\beta_0^{\vphantom{-1}}\beta_1^{-1}$ is slice as a link.
We note that if two string links are concordant then so are their
closures, but the converse does not hold in general; for an in-depth
study related to this, the readers are referred
to~\cite{Habegger-Lin:1998-1}.

Note that our definitions are also meaningful in the topological
category with locally flat submanifolds.  In particular the notion of
a topologically slice string link is defined.

We say that a string link is \emph{$n$-positive}, \emph{$n$-negative}
or \emph{$n$-bipolar} if its closure is, respectively, $n$-positive,
$n$-negative or $n$-bipolar.  Equivalently, these notions can be
defined by asking whether a string link is slice in a $4$-manifold
$V$ with $\partial V = \partial (D^2 \times I\times I)$, where $V$
should satisfy the properties of
Definition~\ref{definition:homotopy-n-positon}.

As in the introduction, we denote the subgroup of topologically slice
$n$-bipolar string links with $m$-components by~$\TSL_n(m)$.  Note
that $\TSL_n(m)$ is closed under group operations by
Theorem~\ref{lemma:basic-properties-of-Z_p-positon}
(\ref{theorem-case:positivity-mirror-image}),
(\ref{theorem-case:positivity-changing-string-orientation}),
(\ref{theorem-case:positivity-split-union}) and
(\ref{theorem-case:positivity-band-sum}), since the closure of
$\beta^{-1}$ is $-\widehat\beta$ and the closure of $\beta_0\beta_1$
is obtained from $\widehat\beta_0 \sqcup \widehat\beta_1$ by band sum.
Also, $\TSL_n(m)$ is a normal subgroup since $\beta$ and a conjugate
of $\beta$ have concordant closures.

This section is devoted to the following special case of
Theorem~\ref{theorem:string-links-infinite-rank} in the introduction:

\begin{theorem}\label{theorem:2-component-string-links-infinite-rank}
  For any $n \ge 0$ the quotient $\TSL_n(2)/\TSL_{n+1}(2)$ contains a
  subgroup whose abelianization is of infinite rank.  For $n\ge 1$,
  the subgroup is generated by string links with unknotted components.
\end{theorem}

For $n=0$, the theorem follows from the result for knots in
\cite{Cochran-Horn:2012-1} by taking the disjoint union of (string
link representations of) the Cochran-Horn knots with a trivial strand.
Theorem~\ref{theorem:string-links-infinite-rank} for the $m>2$
component case follows by adjoining the correct number of trivial
strands.

\subsection{Covering string links}

To investigate the group structure of the string link concordance
group, we formulate a string link version of the covering link
calculus.  For this purpose, again similarly to the link case, it is
natural to consider string links in a $\Z_{(p)}$-homology $D^2\times
I$, which we will call \emph{$\Z_{(p)}$-string links}.  Here a
3-manifold $Y$ is said to be a $\Z_{(p)}$-homology $D^2\times I$ if
$H_*(Y;\Z_{(p)})\cong H_*(D^2\times I;\Z_{(p)})$ and the boundary
$\partial Y$ is identified with $\partial (D^2\times I)$.  The
boundary identification enables us to define product, closure, and
$\Z_{(p)}$-homology concordance of $\Z_{(p)}$-string links.  In
addition, we define the $\Q/\Z$-valued self-linking number of a
component of a $\Z_{(p)}$-string link to be that of the corresponding
component of the closure.  We remark that all $\Z_{(p)}$-string links
considered in this section have components with vanishing
$\Q/\Z$-valued self-linking.

A string link is defined to be $\Z_{(p)}$-homology
\emph{$n$-positive}, \emph{$n$-negative} or \emph{$n$-bipolar} if its
closure is $\Z_{(p)}$-homology \emph{$n$-positive},
\emph{$n$-negative} or \emph{$n$-bipolar}, respectively.  The
definitions of the bipolar heights $\BH(\beta)$ and $\BH^p(\beta)$
carry over verbatim from the ordinary link case.

For a $\Z_{(p)}$-string link $\beta$, we consider the following
operations: (CL1) taking a sublink of $\beta$, and (CL2) taking the
pre-image of $\beta$ in the $p^a$-fold cyclic cover of the ambient
space branched along a component of $\beta$ with vanishing
$\Q/\Z$-valued self-linking.  The result can always be viewed as a
$\Z_{(p)}$-string link; for this purpose we fix, in (CL2), an
identification of the $p^a$-fold cyclic branched cover of $(D^2,[m])$
branched along $i\in [m]$ with $(D^2,[p^a(m-1)+1])$.

\begin{definition}
  \label{definition:covering-string-link}
  A string link obtained from $\beta$ by a finite
  sequence of (CL1) and/or (CL2) is called a \emph{$p$-covering string
    link of $\beta$ of height $\le h$} (or \emph{of height $h$} as an
  abuse of terminology), where $h$ is the number of~(CL2) operations.
\end{definition}

The following immediate consequence of the definitions will be useful.
Note that a fixed sequence of operations (CL1) and (CL2) starting from
an $m$-component string link gives rise to a covering string link of
any $m$-component string link, which we refer to as a
\emph{corresponding} covering string link.

\begin{lemma}
  \label{lemma:covering-string-link-and-product}
  The product operation of string links commutes with covering string
  link operations.  In other words, a covering string link of a
  product is the product of the corresponding covering string links.
\end{lemma}

Essentially, Lemma~\ref{lemma:covering-string-link-and-product} says
that the covering string link operation induces a \emph{homomorphism}
of the string link concordance groups.

\subsection{Computation for string link examples}

To show Theorem~\ref{theorem:2-component-string-links-infinite-rank}
we consider the subgroup generated by the 2-component string links
$C_n(K_i)$, constructed in Section \ref{section:main-examples}, where
$K_i$ are the knots used to
prove~\cite[Theorem~1.1]{Cochran-Horn:2012-1}.  Indeed, for most of
the proof it suffices to assume that the $K_i$ are knots which are
topologically slice, 0-bipolar, and 1-positive.  The exception to this
is that in the last part of
Section~\ref{subsection:proof-of-theorem-infinite-rank} we need the
$K_i$ to also satisfy a technical condition on certain $d$-invariants
of double covers, which was shown in~\cite{Cochran-Horn:2012-1} (see
Proposition~\ref{proposition:cochran-horn-knot-d-invariant}).  As
before we often regard a knot as a string link with one component and
vice versa.

A typical element in the subgroup generated by the $C_n(K_i)$ is of
the form $\prod_{j=1}^s C_n(K_{i_j})^{\eps_j}$ where
$\epsilon_{j}=\pm 1$.

\begin{lemma}\label{lemma:string-links-bipolarity-of-C-doubling}
  The string link $\prod_{j=1}^s C_n(K_{i_j})^{\eps_j}$ has $\BH \ge
  n$, $\BH^p\ge n$ for any prime $p$ and is topologically slice.
\end{lemma}

\begin{proof}
  Each factor string link $C_n(K_{i_j})^{\eps_{j}}$ has closure which
  is $n$-bipolar, by Lemma~\ref{lemma:bipolarity-of-C-doubling}.  If
  $\eps_{j}=-1$ then this still holds by
  Theorem~\ref{theorem:basic-construction-and-positivity}~(\ref{theorem-case:positivity-mirror-image})
  and~(\ref{theorem-case:positivity-changing-string-orientation}). The
  closure of the product of the $C_n(K_{i_j})^{\eps_{j}}$ is a band
  sum of the closures of the~$C_n(K_{i_j})^{\eps_{j}}$.  The latter
  link is $n$-bipolar by
  Theorem~\ref{theorem:basic-construction-and-positivity}~(\ref{theorem-case:positivity-band-sum}).

  Similarly, $\prod_{j=1}^s C_n(K_{i_j})^{\eps_{j}}$ is topologically
  slice and $\Z_{(p)}$-homology $n$-bipolar.
\end{proof}

The remaining part of this section is devoted to showing that the
subgroup generated by the $C_n(K_i)$ has infinite rank abelianization.
It turns out that for this purpose we need a more complicated
application of the covering link calculus than we used in
Section~\ref{section:main-examples}.  In fact this is related to the
orientation reversing which was performed in the last paragraph of the proof of
Theorem~\ref{theorem:height-one-covering-for-main-example}.  This is
not allowed for string links if we want our covering links to respect
the product structure, as in Lemma~\ref{lemma:covering-string-link-and-product}.

To describe our covering string link calculation, we use the following
notation.  For a string link $\beta$, define $r(\beta)$ to be $\beta$
with reversed string orientation, and define $r_s(\beta)=r(\cdots
(r(\beta)) \cdots)$ where $r$ is applied $s$ times.  Note that
$r_s(\beta)=\beta$ if $s$ is even, whereas $r_s(\beta)=r(\beta)$ if $s$ is odd.
Define $T(\beta)$ to be the string link shown in
Figure~\ref{figure:string-link-twisting}, and
$T_s(\beta)=T(\cdots(T(\beta)) \cdots)$ where $T$ is applied $s$
times. For a 1-component string link $\alpha$, let
$\ell_2(\alpha)=\ell(\alpha,\alpha)$ be the split union of two copies
of~$\alpha$, which is viewed as a 2-component string link.  Recall
from Section~\ref{section:main-examples} that $\tilde\beta$ is the
1-component string link shown in Figure~\ref{figure:plat-closure}.
Define $N_d(J)$ to be the $d$-fold cyclic branched cover of a
knot~$J$.  We denote $N_d(\widehat\alpha)$ by $N_d(\alpha)$ for a
1-component string link $\alpha$, viewing $\alpha$ as a knot.  Given a
string link $\beta$ in $Y$ and another 3-manifold $N$, remove a 3-ball
from $Y$ which is disjoint to $\beta$.  Filling in this 3-ball with a
punctured $N$, we obtain a new string link in $Y\#N$. We call the
result of this construction ``$\beta$ in the $Y$ summand of $Y \# N$''.

\begin{figure}[H]
  \labellist
  \hair 0mm
  \pinlabel {$\beta$} at 38 13
  \endlabellist
  \includegraphics{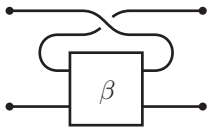}
  \caption{The string link $T(\beta)$.}
  \label{figure:string-link-twisting}
\end{figure}

\begin{theorem}
  \label{theorem:height-one-string-link-covering-for-main-example}
  Suppose $\beta$ is a 2-component string link with unknotted
  components.  Then the string link $T_s(r_t(C(\beta)))\cdot
  \ell_2(\alpha)$ has, as a $p$-covering string link of height one,
  the string link $T_{2s+1}(r_{s+t+1}(\beta))\cdot
  \ell_2(r_{s+t}(\tilde\beta)\cdot \alpha)$ in the $D^2\times I$
  summand of $(D^2\times I) \# N_{p^a}(\alpha)$.
\end{theorem}

See Figures~\ref{figure:doubled-string-link}
and~\ref{figure:covering-doubled-string-link} for the base and
covering string links in
Theorem~\ref{theorem:height-one-string-link-covering-for-main-example}
for $t=0$.  Here, \fbox{$s$} represents $s$ left-handed \emph{half}
twistings (i.e.\ $s$ positive crossings) arranged vertically, which are
obtained by applying~$T_s$.

\begin{figure}[H]
  \labellist
  \hair 0mm
  \pinlabel {$\beta$} at 55 61
  \pinlabel {$s$} at 64.5 117
  \pinlabel {$\alpha$} at 204 135
  \pinlabel {$\alpha$} at 204 10
  \endlabellist
  \includegraphics{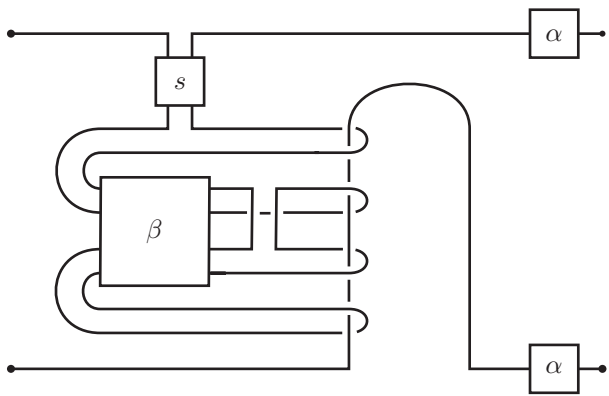}
  \caption{The string link $T_s(C(\beta))\cdot \ell_2(\alpha)$.}
  \label{figure:doubled-string-link}
\end{figure}

\begin{figure}[H]
  \labellist
  \hair 0mm
  \pinlabel {$2s+1$} at 38 59
  \pinlabel {$r_{s+1}(\beta)$} at 38 15
  \pinlabel {$r_s(\tilde\beta)$} at 85 78
  \pinlabel {$r_s(\tilde\beta)$} at 85 10
  \pinlabel {$\alpha$} at 119 77
  \pinlabel {$\alpha$} at 119 9
  \endlabellist
  \includegraphics{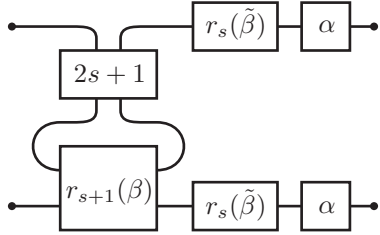}
  \caption{The string link $T_{2s+1}(r_{s+1}(\beta))\cdot
    \ell_2(r_s(\tilde\beta)\cdot \alpha)$.}
  \label{figure:covering-doubled-string-link}
\end{figure}

\begin{proof}[Proof of
  Theorem~\ref{theorem:height-one-string-link-covering-for-main-example}]
  We will give a proof for $t=0$ only, since exactly the same argument
  shows the $t=1$ case; only the residue of $t$ modulo $2$ matters.  By choosing $a$ such that $p^a\ge 5$ and
  taking the $p^a$-fold cyclic branched cover of $D^2\times I$ along
  the first (bottom) component of the link in
  Figure~\ref{figure:doubled-string-link}, we obtain a $p$-covering
  link which is similar to that shown in
  Figure~\ref{figure:cyclic-branched-covering}.  By taking a
  sublink similar to the sublink which was taken to pass from
  Figure~\ref{figure:cyclic-branched-covering} to
  Figure~\ref{figure:two-component-covering-link}, we obtain the
  covering string link shown in
  Figure~\ref{figure:two-component-covering-string-link}, which is in
  the $D^2\times I$ summand of $(D^2\times I)\# N_{p^a}(\alpha)$.

  \begin{figure}[H]
    {\small
    \labellist
    \hair 0mm
    \pinlabel {$\beta$} at 49 303
    \pinlabel {$\beta$} at 119 239
    \pinlabel {$\beta$} at 189 163
    \pinlabel {$\beta$} at 259 98
    \pinlabel {$\beta$} at 329 22
    \pinlabel {$s$} at 197.5 206
    \pinlabel {$s$} at 337 66
    \pinlabel {$\alpha$} at 267 217
    \pinlabel {$\alpha$} at 407 77
    \endlabellist
    \includegraphics[scale=.8]{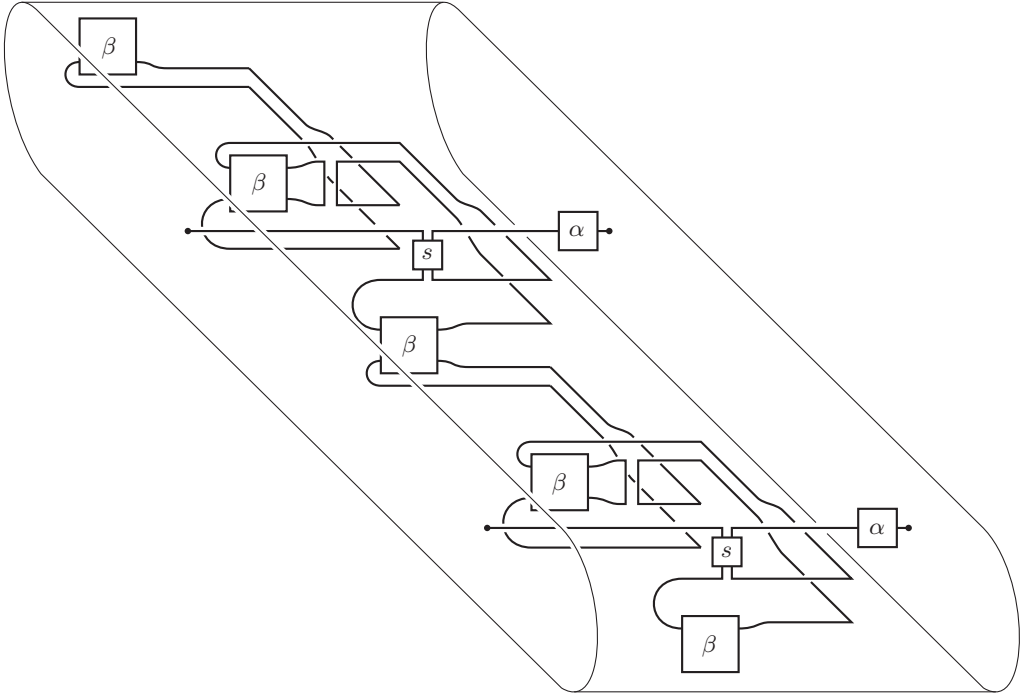}
    }
    \caption{A covering string link of
      Figure~\ref{figure:doubled-string-link}, drawn in $D^2\times
      I$.}
    \label{figure:two-component-covering-string-link}
  \end{figure}

  After an isotopy, we obtain the string link in
  Figure~\ref{figure:two-component-covering-string-link-2}.  We use
  that a single component of $\beta$ is unknotted, so that the upper
  left and lower right occurrences of $\beta$ in
  Figure~\ref{figure:two-component-covering-string-link} are removed.

  \begin{figure}[H]
    \labellist
    \hair 0mm
    \pinlabel {$\beta$} at 104 105
    \pinlabel {$\alpha$} at 183.5 146.5
    \pinlabel {$\alpha$} at 307 34
    \pinlabel {$s$} at 237 23
    \pinlabel {$s$} at 113 135
    \pinlabel {$\widetilde\beta$} at 37 183
    \pinlabel {$\widetilde\beta$} at 157 69
    \endlabellist
    \includegraphics[scale=1]{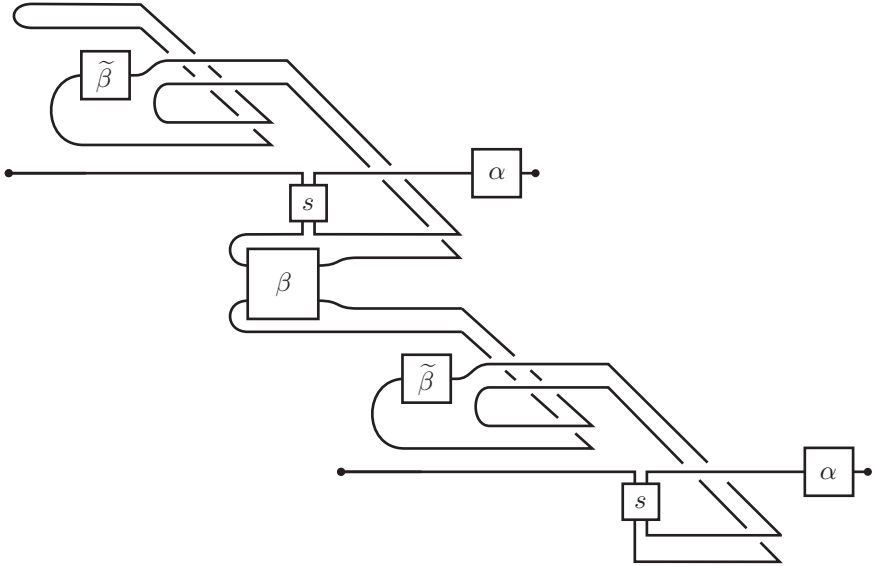}
    \caption{A simplified version of
      Figure~\ref{figure:two-component-covering-string-link}.}
    \label{figure:two-component-covering-string-link-2}
  \end{figure}

  Then a further isotopy gives us
  Figure~\ref{figure:isotoped-covering-string-link}.  Here the upper
  component in
  Figure~\ref{figure:two-component-covering-string-link-2} corresponds
  to the upper component of
  Figure~\ref{figure:isotoped-covering-string-link}.  Using the fact
  that a local knot on a component of a string link can be moved to
  anywhere on the same component, we see that the upper left
  $\widetilde\beta$ in
  Figure~\ref{figure:two-component-covering-string-link-2} becomes the
  upper $r_s(\tilde\beta)$ part in
  Figure~\ref{figure:isotoped-covering-string-link}.  The lower
  component is simplified similarly but an additional half twist is
  introduced.

  \begin{figure}[H]
    \labellist
    \hair 0mm
    \pinlabel {$\beta$} at 38 65
    \pinlabel {$s+1$} at 38 27.5
    \pinlabel {$s$} at 38 104
    \pinlabel {$r_s(\tilde\beta)$} at 76 122
    \pinlabel {$r_s(\tilde\beta)$} at 76 10
    \pinlabel {$\alpha$} at 110 122
    \pinlabel {$\alpha$} at 110 10
    \endlabellist
    \includegraphics{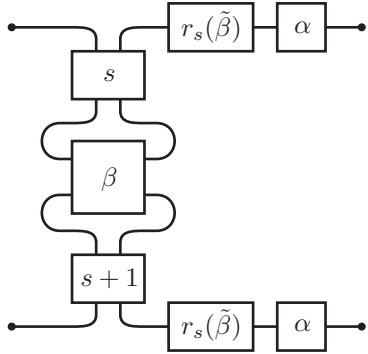}
    \caption{A planar version of
      Figure~\ref{figure:two-component-covering-string-link-2}.}
    \label{figure:isotoped-covering-string-link}
  \end{figure}

  Finally, by moving the box \fbox{$\beta$} down, across the $s+1$ half
  twists, we obtain the desired covering string link
  $T_{2s+1}(r_{s+1}(\beta))\cdot \ell_2(r_s(\tilde\beta)\cdot\alpha)$
  illustrated in Figure~\ref{figure:covering-doubled-string-link}.
\end{proof}

\begin{corollary}
  \label{corollary:iterated-string-link-covering-for-main-example}
  For any 1-component string link $\gamma$, the 2-component string
  link $C_n(\gamma)$ has, as a $p$-covering string link of height $n$,
  the 1-component string link 
  $\gamma \cdot r(\gamma)\cdot
  \widetilde{C_1(\gamma)}\cdots\widetilde{C_{n-1}(\gamma)}$
  in the
  $D^2\times I$ summand of the connected sum of $D^2\times I$ and
  $\Z_{(p)}$-homology spheres of the form
  $N_{p^a}(\widetilde{C_k(\gamma)})$ where either $2\le k\le n-1$ or
  $(k,p^a)=(1,p)$.
\end{corollary}

\begin{proof}
  By repeated application of
  Theorem~\ref{theorem:height-one-string-link-covering-for-main-example}
  starting with $C_n(\gamma)$, we obtain
  \begin{multline*}
    T_1(r(C_{n-1}(\gamma)))\ell_2(\widetilde{C_{n-1}(\gamma)}),\;\;
    T_3(r(C_{n-1}(\gamma)))\ell_2(\widetilde{C_{n-2}(\gamma)}\widetilde{C_{n-1}(\gamma)}), \ldots, \\
    T_{2^{n-1}-1}(r(C_1(\gamma)))\ell_2(\widetilde{C_1(\gamma)}\cdots\widetilde{C_{n-1}(\gamma)})
  \end{multline*}
  as covering string links of~$C_n(\gamma)$.  The last covering string link in the above list has
  height $n-1$ and is in the $D^2\times I$ summand of $(D^2\times I)\#
  N$, where $N$ is a connected sum of $\Z_{(p)}$-homology spheres of
  the form $N_{p^a}(\widetilde{C_k(\gamma)})$ with $2\le k \le n-1$.
  This covering string link is shown in
  Figure~\ref{figure:bing-doubled-string-link}.

  \begin{figure}[H]
    \labellist
    \hair 0mm
    \pinlabel {$r(\gamma)$} at 47 32
    \pinlabel {$2^{n-1}-1$} at 47 78
    \pinlabel {$\widetilde{C_1(\gamma)}\cdots \widetilde{C_{n-1}(\gamma)}$} at 131 95
    \pinlabel {$\widetilde{C_1(\gamma)}\cdots \widetilde{C_{n-1}(\gamma)}$} at 131 9
    \endlabellist
    \includegraphics{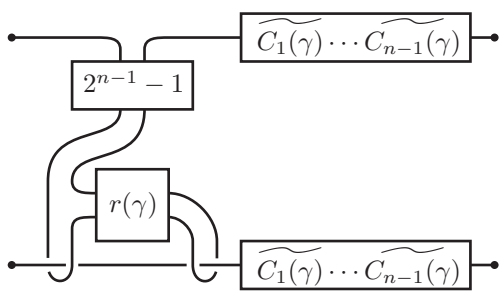}
    \caption{A height $n-1$ covering string link of $C_n(\gamma)$.}
    \label{figure:bing-doubled-string-link}
  \end{figure}

  As the final covering, we proceed similarly to
  \cite[Section~3]{Cha-Livingston-Ruberman:2006-1}.  By taking the
  $p$-fold cover branched along the bottom component of the link in
  Figure~\ref{figure:bing-doubled-string-link} and then taking a
  component, we obtain the string link illustrated in
  Figure~\ref{figure:final-covering-string-link}, which is in the
  $D^2\times I$ summand of $(D^2\times I)\#N_1$, where $N_1$ is a
  connected sum of $\Z_{(p)}$-homology spheres of the form
  $N_{p^a}(\widetilde{C_k(\gamma)})$ for either $2 \le k \le n-1$ or
  $(k,p^a)=(1,p)$.  Note that this can be done even for $p=2$; the
  final covering does not require us to take a covering of order $p^a
  \ge 5$.  It is
  easily seen that the link in
  Figure~\ref{figure:final-covering-string-link} is isotopic to
  $\gamma \cdot r(\gamma) \cdot  \widetilde{C_1(\gamma)}\cdots\widetilde{C_{n-1}(\gamma)}$ as
  desired.
\end{proof}

\begin{figure}[H]
  \labellist
  \hair 0mm
  \pinlabel {$\widetilde{C_1(\gamma)}\cdots \widetilde{C_{n-1}(\gamma)}$} at 170 100
  \pinlabel {$r(\gamma)$} at 24 60
  \pinlabel {$r(\gamma)$} at 96 37
  \pinlabel {$2^{n-1}-1$} at 95 83
  \endlabellist
  \includegraphics{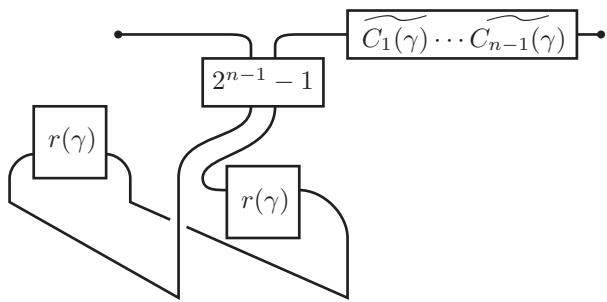}
  \caption{A height $n$ covering string link of $C_n(\gamma)$.}
  \label{figure:final-covering-string-link}
\end{figure}

\subsection{Proof of
  Theorem~\ref{theorem:2-component-string-links-infinite-rank}}
\label{subsection:proof-of-theorem-infinite-rank}

We are now ready to prove Theorem
\ref{theorem:2-component-string-links-infinite-rank}.  As before,
consider the 2-component string link $\beta = \prod_{j=1}^s
C_n(K_{i_j})^{\eps_j}$ where the $K_i$ are the knots in
\cite{Cochran-Horn:2012-1}.  We have that $\beta$ is $n$-bipolar and
topologically slice by
Lemma~\ref{lemma:string-links-bipolarity-of-C-doubling}.  Therefore,
it suffices to show the following: define $a_i := \sum_{\{j | i_j =
  i\}} \eps_{j}$.  Suppose $a_i\ne 0$ for some~$i$.  Then $\beta$ is
not $\Z_{(p)}$-homology $(n+1)$-bipolar.

In this proof we will work with a general prime $p$ for as long as
possible, although at the end of the proof we will specialize to
$p=2$.  The specialization occurs due to the fact that we need to use $d$-invariant
calculations which were made for double covers of knots, by
Manolescu-Owens~\cite{Manolescu-Owens:2007} and
Cochran-Horn~\cite{Cochran-Horn:2012-1}.

By replacing $\beta$ with $\beta^{-1}$ and reindexing the knots $K_i$
if necessary, we may assume that $a_1 > 0$.  Suppose for a
contradiction that $\beta$ is $\Z_{(p)}$-homology $(n+1)$-negative.
Applying Theorem~\ref{theorem:covering-positon-theorem}, we see that the knot
\[
\bigconnsum_{j=1}^s \eps_{j}\bigg(K_{i_j} \# K_{i_j}^r \#
\Big(\bigconnsum_{k=1}^{n-1} \wt{C_{k}(K_{i_j})}\Big)\bigg)
\]
is $\Z_{(p)}$-homology 1-negative, since it is the closure of a height
$n$ covering string link of $\beta$ by
Corollary~\ref{corollary:iterated-string-link-covering-for-main-example}
and Lemma~\ref{lemma:covering-string-link-and-product}.  Here, by the
description of the ambient space of our covering links in
Corollary~\ref{corollary:iterated-string-link-covering-for-main-example},
the knot lies in a 3-ball in the connected sum, say $N$, of
3-manifolds of the form $N_{p^a}(\eps_{j}\widetilde{C_{k}(K_{i_j})})$,
where either $k\ge 2$ or $(k,p^a)=(1,p)$; we note that the exponent
$a$ need not be the same for different summands.

Recall that $K_{i_j}$ is 0-bipolar.  So $C_k(K_{i_j})^{\eps_{j}}$ is
$2$-bipolar for $k\ge 2$ by
Lemma~\ref{lemma:bipolarity-of-C-doubling}.  Since the plat closure
$(-)\sptilde$ is obtained from the closure $(-)\sphat$ by band sum of
components, $\eps_{j}\widetilde{C_k(K_{i_j})}$ is $2$-bipolar for $k
\ge 2$ by
Theorem~\ref{theorem:basic-construction-and-positivity}~(\ref{theorem-case:positivity-band-sum}).
Similarly $\eps_{j}\widetilde{C_1(K_{i_j})}$ is $1$-bipolar.
Therefore $-\eps_{j}\widetilde{C_k(K_{i_j})}$ is $1$-negative for $k
\geq 1$.  By Theorem
\ref{theorem:basic-construction-and-positivity}~(\ref{theorem-case:positivity-band-sum}),
we have that the knot
  \[
  \bigconnsum_{j=1}^s \eps_{j}\bigg(K_{i_j} \# K_{i_j}^r \#
  \Big(\bigconnsum_{k=1}^{n-1} \wt{C_{k}(K_{i_j})}\#
  -\wt{C_{k}(K_{i_j})}\Big)\bigg),
  \]
  which is again in a 3-ball in $N$, is $\Z_{(p)}$-homology
  $1$-negative.  Thus, by
  Theorem~\ref{theorem:basic-construction-and-positivity}~(\ref{theorem-case:positivity-slice-link}),
  the knot
  \[
  \bigconnsum_{j=1}^s \eps_{j} (K_{i_j} \# K_{i_j}^r),
  \]
  which is in a 3-ball in $N$, is $\Z_{(p)}$-homology $1$-negative.
  Since the connected sum operation is commutative, the knot
  \[
  J:= \bigconnsum_i a_i (K_i \# K_i^r)
  \]
  lying in a 3-ball in $N$ is $\Z_{(p)}$-homology $1$-negative.

  We will remove many summands from the ambient space $N$ of $J$,
  without altering the $\Z_{(p)}$-homology $1$-negativity.  Let $V$ be
  a $\Z_{(p)}$-homology $1$-negaton for~$J$.  First, for $k \ge 2$, we
  will remove all the $N_{p^a}(\eps_{j}\widetilde{ C_{k}(K_{i_j})})$
  summands.  Recall that $\widetilde{\eps_jC_k(K_{i_j})}$ is 2-bipolar for $k\ge
  2$.  So, by Theorem~\ref{theorem:covering-positon-theorem}, the
  branched cover $N_{p^a}(\eps_{j}\widetilde{C_k(K_{i_j})})$ bounds
  a $\Z_{(p)}$-homology 1-positon, say~$V_{i_j}$.  View $\partial V$
  as the union of a (many) punctured $S^3$ and a disjoint union of
  punctured $N_{p^a}(\eps_{j}\widetilde{C_k(K_{i_j})})$s glued along
  the boundary.  Viewing each $V_{i_j}$ as a (relative to the
  boundary) cobordism from a punctured
  $N_{p^a}(\eps_{j}\widetilde{C_k(K_{i_j})})$ to $B^3$, and attaching
  each $-V_{i_j}$ to $V$ along the punctured
  $N_{p^a}(\eps_{j}\widetilde{C_k(K_{i_j})})$ for $k\ge 2$, we obtain
  a 4-manifold $W$ whose boundary is a connected sum, say $N'$, of
  3-manifolds of the form $N_p(\eps_{j}\widetilde{C_1(K_{i_j})})$.
  Moreover, $W$ is a $\Z_{(p)}$-homology $1$-negaton for the knot $J$
  which is now considered as a knot in~$N'$.  This can be seen by an
  easy Mayer-Vietoris argument which shows that
  \[
  \textstyle
  H_2(W)/\text{torsion} \cong
  \big(H_2(V)/\text{torsion}\big) \oplus \big(\bigoplus
  H_2(V_{i_j})/\text{torsion}\big),
  \]
  and by observing that the change in orientation of the $V_{i_j}$
  causes their intersection forms to be negative definite.

  The argument above for removing the $N_{p^a}(\eps_{j}\widetilde{
    C_{k}(K_{i_j})})$ summands for $k\ge 2$ also works in the case
  $k=1$, when $\eps_{j} =1$, since $K_{i_j}$ is 1-positive and so
  $\widetilde{C_1(K_{i_j})}$ is $2$-positive.  So we can in fact
  assume that $J$ is a $\Z_{(p)}$-homology 1-negative knot in a 3-ball
  in a connected sum of 3-manifolds of the form
  $N_{p}(-\widetilde{C_{1}(K_{i_j})})$.

  Furthermore, since $a_1>0$ and $K_1$ is $1$-positive by our
  assumption, it follows that
  \[
  J' = K_1 \# \Big(\bigconnsum_{i\ne 1} a_i (K_i^{\vphantom{r}}\# K_i^r)\Big)
  \]
  which is in the same ambient space is $\Z_{(p)}$-homology
  $1$-negative.

  We will derive a contradiction by using $d$-invariants.  We remark
  that we essentially follow the argument in
  \cite[Section~4,~Proof~of~Theorem~1.1]{Cochran-Horn:2012-1}, with
  additional complication required to resolve the difficulty from
  the remaining 3-manifold summands in the ambient space of~$J'$.

  From now on we restrict to $p=2$.  Let $\Sigma$ be the double
  branched cover of~$J'$.  It is easily seen that $\Sigma$ is the
  connected sum of $N_2(K_1)$, $a_iN_2(K_i\#K_i^r)$ with $i\ne 1$, and
  additional summands of the form $N_2(-\widetilde{C_1(K_{i_j}^r)})$.
  By combining Figures~\ref{figure:knot-doubling}
  and~\ref{figure:plat-closure}, observe that the knot
  $-\widetilde{C_1(K_{i_j}^r)}$ is the Whitehead double of
  $-K_{i_j}^r$ with positive clasp.  So
  $N_2(-\widetilde{C_1(K_{i_j}^r)})$ is a homology sphere.  It follows
  that $N_2(-\widetilde{C_1(K_{i_j}^r)})$ has a unique spin$^c$
  structure, and the spin$^c$ structures of
  \[
  Y:= N_2(K_1) \# \Big(\bigconnsum_{i\ne 1} a_i N_2(K_i\#K_i^r) \Big)
  \]
  are in 1-1 correspondence with those of~$\Sigma$.  We need the
  following fact concerning the knots~$K_i$:

  \begin{proposition}[{\cite[Section~4,~Proof~of~Theorem~1.1]{Cochran-Horn:2012-1}}]
    \label{proposition:cochran-horn-knot-d-invariant}
    There is a spin$^c$ structure $\mathfrak{s}$ on $Y$ such that
    $d(Y,\mathfrak{s}) < 0$ and the corresponding first homology class
    lies in any metabolizer.
  \end{proposition}

  In our case, the corresponding spin$^c$ structure $\mathfrak{t}$ of
  $\Sigma$ also represent a homology class lying in any metabolizer,
  and we have $d(\Sigma,\mathfrak{t})=d(Y,\mathfrak{s}) + \sum_\ell
  d(Y_\ell)$ where the $Y_\ell$ denote the
  $N_2(-\widetilde{C_1(K_{i_j}^r)})$ summands and $d(Y_\ell)$ denotes
  the $d$-invariant associated to the unique spin$^c$-structure.  By
  \cite[Theorem~1.5]{Manolescu-Owens:2007}, $d(Y_\ell)\le 0$ since
  $-\widetilde{C_1(K_{i_j}^r)}$ is a positive Whitehead double.  It
  follows that $d(\Sigma,\mathfrak{t})<0$.  By the $\Z_{(2)}$-homology
  1-negative version of Theorem~\ref{theorem:d-invariant-obstruction},
  this contradicts that $J'$ is $\Z_{(2)}$-homology 1-negative.
  \qed

\appendix

\section{Signature invariants and 0-positivity of knots in
  $\Z_{(p)}$-homology spheres}
\label{section:signature-and-0-positivity}

In this appendix we give a proof that the signature invariant of knots
in rational homology spheres defined in~\cite{Cha-Ko:2000-1} gives an
obstruction to $\Z_{(p)}$-homology $0$-positivity.  This is a
generalization of \cite[Proposition~4.1]{Cochran-Harvey-Horn:2012-1}

We begin by describing the invariant following~\cite{Cha-Ko:2000-1}.
Suppose $K$ is a knot in a rational homology 3-sphere~$Y$.  A surface
$F$ embedded in $Y$ is called a \emph{generalized Seifert surface} for
$K$ if for some $c\ne 0$, $F$ is bounded by the union of $c$ parallel
copies of $K$ which are taken along a framing on $K$ which agrees with
the framing induced by $F$ on each parallel copy.  The integer $c$ is
called the \emph{complexity} of~$F$.  See Lemma
\ref{lemma:positivity-zero-framing} and the preceding discussion on
zero-framings of knots in rational homology spheres.  A Seifert matrix
$A$ is defined as usual: choosing a basis $\{x_i\}$ of $H_1(F)$,
$A=(\lk_Y(x_i^+,x_j))_{ij}$ where $x_i^+$ is obtained by pushing $x_i$
slightly along the positive normal direction of~$F$.  Note that here
the linking number $\lk_Y$ is rational-valued.  Define, for $\theta\in
\R$,
\[
\sigma_A(\theta)=\sign\big( (1-e^{2\pi i \theta}) A + (1-e^{-2\pi i
  \theta}) A^T \big)
\]
and let
\[
\overline\sigma_A(\theta)=\frac12 \Big(\lim_{\phi\to
  \theta_-}\sigma_A(\phi)+\lim_{\phi\to \theta_+}\sigma_A(\phi)\Big)
\]
be the average of the one-sided limits.  Now the \emph{signature
  average function for $K$} is defined by
$\overline\sigma_K(\theta)=\overline\sigma_A(\theta/c)$, where $c$ is
the complexity of the generalized Seifert surface~$F$.  Due to
\cite{Cha-Ko:2000-1}, the function $\overline\sigma_K\colon \R\to \Z$
is invariant under a concordance in a rational homology
$S^3\times I$.\footnote{In~\cite{Cha-Ko:2000-1}, they consider the
  jump at $\theta$, rather than the average, as a concordance
  invariant.  It is easy to see that the average function determines the
  jump function and vice versa.}  For knots in $S^3$,
$\overline\sigma_K$ is equal to the Levine-Tristram signature.

\begin{theorem}
  If $K$ is $\Z_{(p)}$-homology $0$-positive, then
  $\bar\sigma_K(\theta)\le 0$ for any $\theta\in \R$.
\end{theorem}

\begin{proof}
  First note that $\bar\sigma_K$ is defined by
  Lemma~\ref{lemma:positivity-zero-framing}.  Suppose $V$ is a
  $\Z_{(p)}$-homology $0$-positon for an $K\subset Y$, with slicing
  disk~$\Delta$.  Choose a generalized Seifert surface $F$ for~$K$,
  and let $A$ be the Seifert matrix.  It suffices to show that
  $\sigma_A(\theta)\le 0$ on a dense subset of~$\R$.  By using the
  observation in \cite[p.~1178]{Cha-Ko:2000-1} that the Seifert
  pairing vanishes at $(x,y)$ whenever either $x$ or $y$ is a boundary
  component of $F$, it can be seen that the Seifert matrix of $A$ is
  $S$-equivalent to the Seifert matrix of a new surface obtained by
  attaching half-twisted bands to the boundary of $F$ in such a way
  that the new boundary is the $(n,1)$-cable of~$K$.  Since the
  $(n,1)$-cable of $K$ is $\Z_{(p)}$-homology $0$-positive by
  Theorem~\ref{theorem:basic-construction-and-positivity}~(\ref{theorem-case:positivity-generalize-satellite}),
  we may assume that $\partial F=K$.

  We will use the following fact: consider a rational homology
  3-sphere $\Sigma$ and a framed knot $J$ in~$\Sigma$.  Suppose
  $(W,M)$ is a $(4,2)$-manifold pair with $M$ framed (i.e.\ we
  identify $\nu(M)$ with $M\times D^2$), such that $(\Sigma,J)$ is a
  component of $\partial(W,M)$, $\tilde H_*(\partial W;\Q)=0$, and
  $(W,M)$ is primitive in the sense of \cite[p.~1170]{Cha-Ko:2000-1},
  that is, there is a homomorphism $H_1(W-M)\to \Z$ whose restriction
  on the circle bundle of $M$ coincides with $H_1(M\times S^1) \to
  H_1(S^1)=\Z$ induced by the projection.  Let $\tilde W$ be the
  $d$-fold cyclic branched cover of $W$ along $M$, and $t\colon \tilde
  W\to \tilde W$ be the generator of the covering transformation group
  corresponding to the (positive) meridian of~$J$.  Let
  $\sigma_{k,d}(W,M)$ be the signature of the restriction of the
  intersection pairing on the $e^{2\pi i k/d}$-eigenspace of
  $t_*\colon H_2(\tilde W;\C) \to H_2(\tilde W;\C)$.

  \begin{lemma}
    \label{lemma:invariance-of-eigenspace-signature-defect}
    The value $\sigma_{k,d}(W,M)-\sign(W)$ is determined by
    $(\Sigma,J)$, independently of the choice of $(W,M)$.
  \end{lemma}

  Some special cases of
  Lemma~\ref{lemma:invariance-of-eigenspace-signature-defect}, at
  least, seem to be folklore in the classical signature theory (e.g.,
  see \cite[Theorem~4.4]{Viro:1973-1} for the case of
  $W=\Sigma\times I$).  The above general case is essentially
  proven by an argument of \cite[Lemma~4.2]{Cha-Ko:2000-1}, although
  \cite[Lemma~4.2]{Cha-Ko:2000-1} is stated with a slightly stronger
  hypothesis to eliminate the $\sign(W)$ term.  For the reader's
  convenience we will give a proof later.

  Let $F'\subset Y\times I$ be the framed surface obtained by
  pushing the interior of $F\subset Y=Y\times 0$ slightly into
  $Y\times(0,1)$.  In \cite[Lemma~4.3]{Cha-Ko:2000-1}, it was shown
  that $\sigma_A(k/d)=\sigma_{k,d}(Y\times I, F')$.  Since
  $\sign(Y\times I)=0$, it follows that $\sigma_A(k/d) =
  \sigma_{k,d}(V,\Delta)-\sign(V)$ by
  Lemma~\ref{lemma:invariance-of-eigenspace-signature-defect}.

  The remaining part of the proof is now almost identical with that of
  \cite[Proposition~4.1]{Cochran-Harvey-Horn:2012-1}.  Since $V$ has
  positive definite intersection form, $\sign(V)=b_2(V)$.  Using the
  observations given above, the eigenspace-refined Euler
  characteristic argument of the last part of the proof of
  \cite[Proposition~4.1]{Cochran-Harvey-Horn:2012-1} is carried out to
  show that $\sigma_{k,d}(V,\Delta)\le b_2(V)$ whenever $d\nmid k$.
  Therefore $\sigma_A(k/d) \le 0$ for $d\nmid k$.  Since such $k/d$
  form a dense subset of $\R$, the proof is completed.
\end{proof}

\begin{proof}[Proof of Lemma~\ref{lemma:invariance-of-eigenspace-signature-defect}]
  Suppose both $(W,M)$ and $(W',M')$ are as in the description above
  the statement of
  Lemma~\ref{lemma:invariance-of-eigenspace-signature-defect}.  We
  consider the pair $(X,E)=(W,M)\cup_{(\Sigma,J)} -(W',M')$.  We
  identify the normal bundle of $E$ with $E\times D^2$ under the
  framing of~$E$.  By the assumption there is a map $X-E \to S^1$
  which restricts to the projection $E\times S^1 \to S^1$ on the
  circle bundle associated to the normal bundle of~$E$.  We may assume
  that it restricts to a constant map on $\partial X$, since
  $H_1(\partial X)$ is torsion.  Extend it to $f\colon X=X\times 0\to
  D^2$ by glueing the projection $E\times D^2 \to D^2$, and extend $f$
  to $g\colon U=X\times I \to D^2$ in such a way that the restriction
  of $g$ on $(\partial X \times I) \cup_{\partial X\times 1} (X\times
  1)$ is a constant map away from $0\in D^2$.  We may assume $g$ is
  smooth and $0\in D^2$ is a regular value.  Now $N=g^{-1}(0)$ is a
  submanifold in $U$ satisfying
  \[
  \partial(U,N)=(X\times 0,E)\cup(\partial X\times I,\emptyset)
  \cup (X\times 1,\emptyset).
  \]
  It follows that $\sigma_{k,d}(X,E)+\sigma_{k,d}(\partial X\times
  I,\emptyset)=\sigma_{k,d}(X,\emptyset)$.  It is well known that
  $\sigma_{k,d}(X,\emptyset)=\sign(X)$, namely, the eigenspace
  refined signature of the $d$-fold covering associated to the zero
  map $\pi_1(-)\to \Z_d$ is equal to the ordinary signature.  We also
  have $\sigma_{k,d}(\partial X\times I,\emptyset) = \sign(\partial
  X\times I)=0$, where the last equality holds since $\partial X\times
  I$ is a product.  By combining the above with Novikov additivity, it
  follows that
  \[
  \sigma_{k,d}(W,M)-\sigma_{k,d}(W',M') = \sigma_{k,d}(X,E) = \sign(X)
  = \sign(W)-\sign(W'). \qedhere
  \]
\end{proof}

\section{Amenable von Neumann $\rho$-invariants and $n$-positivity}
\label{section:rho-invariant-and-positivity}

In this appendix we discuss the relationship between amenable von
Neumann $\rho$-invariants and $\Z_{(p)}$-homology $n$-positivity.

\begin{theorem}
  \label{theorem:rho-invariant-and-positivity}
  Suppose a knot $K$ has a $\Z_{(p)}$-homology $n$-positon $V$ with
  slicing disk~$\Delta$.  Let $M(K)$ be the zero-surgery manifold
  of~$K$.  If $\phi\colon\pi_1(M(K))\to\Gamma$ is a homomorphism into
  an amenable group $\Gamma$ lying in Strebel's class $D(\Z_p)$ that
  sends a meridian of $K$ to an infinite order element and extends to
  $\psi\colon \pi_1(V-\Delta)\to \Gamma$, then the von Neumann
  invariant $\rhot(M(K),\phi)$ is nonpositive.  In addition, if
  $\psi(\cP^n \pi_1(V-\Delta))=\{e\}$ \textup{(}for example when
  $\cP^n \Gamma=\{e\}$\textup{)}, then $\rhot(M(K),\phi)=0$.
\end{theorem}

Our proof is a combination of the idea of its homotopy $n$-positive
analogue (see \cite[Theorem~5.8]{Cochran-Harvey-Horn:2012-1}) and the
technique for amenable $L^2$-invariants in~\cite{Cha:2010-01}.

\begin{proof}
  Consider $W_0=V-\nu(D)$ whose boundary is~$M(K)$.  The invariant
  $\rhot(M(K),\phi)$ is given by (or defined to be) the
  $L^2$-signature defect $\lsign_\Gamma(W_0)-\sign(W)$, where
  $\lsign_\Gamma(W_0)$ is the $L^2$-signature of the intersection
  pairing defined on $H_2(W_0;\N \Gamma)$ and $\N \Gamma$ is the group
  von Neumann algebra of~$\Gamma$.  Since $\Gamma$ is amenable and in
  $D(\Z_p)$, we have that the $L^2$-dimension $\ldim H_2(W_0;\N G)$ is
  not greater than the $\Z_p$-Betti number $b_2(W_0;\Z_p)$ by
  \cite[Theorem~3.11]{Cha:2010-01}.  It follows that
  $|\lsign_\Gamma(W_0)| \le b_2(W_0;\Z_p)$.  By
  Lemma~\ref{lemma:basic-properties-of-Z_p-positon}~(\ref{lemma-case:H2-of-positon}),
  $b_2(W_0;\Z_p)=b_2(W_0)$.  Since $W_0$ has positive definite
  intersection form, $b_2(W_0)=\sign(W_0)$.  It follows that
  $\rhot(M(K),\phi)\le 0$.

  To prove the remaining conclusions of the theorem, consider the
  connected sum $W$ of $W_0$ and $r$ copies of $\overline{\C P^2}$,
  where $r=b_2(V)$.  Then $W$ has the following properties:
  \begin{enumerate}
  \item $\partial(W)=M(K)$.
  \item $H_1(W)\cong\Z \oplus (\text{torsion coprime to $p$})$, and
    the meridian of $K$ represents an integer coprime to $p$ in
    $H_1(W)/\text{torsion}\cong \Z$.
  \item $H_2(W)=\Z^{2r}\oplus (\text{torsion coprime to $p$})$.  There
    exist elements $\ell_1,\ldots,\ell_r, d_1,\ldots, d_r \in
    H_2(W;\Z[\pi_1(W)/\cP^n \pi_1(W)])$ such that the images of the
    $\ell_i$, $d_j$ generate $H_2(W)$ modulo torsion,
    $\lambda_n(\ell_i,\ell_j)=0$ and
    $\lambda_n(\ell_i,d_j)=\delta_{ij}$, where $\lambda_n$ is the
    intersection pairing over $\Z[\pi_1(W)/\cP^n \pi_1(W)]$.
  \end{enumerate}
  (1) is obvious. (2) and the first statement of (3) follow
  immediately from Lemma~\ref{lemma:basic-properties-of-Z_p-positon}.
  Consider the surfaces $S_j$ given in the definition of
  $\Z_{(p)}$-homology $n$-positivity
  (Definition~\ref{definition:homology-n-positon}), and the surfaces
  $P_i =$ ($\C P^1$ in the $i$th $\overline{\C P^2}$ summand).  We may
  assume both $S_j$ and $P_i$ lie in~$W$.  Since $\pi_1(S_j)$,
  $\pi_1(P_i) \subset \cP^n \pi_1(W)$, the classes of $S_j$ and $P_i$
  are in $H_2(\Z[\pi_1(W)/\cP^n \pi_1(W)])$.  We have
  $\lambda_n([S_i],[S_j])=\delta_{ij}$, $\lambda_n([S_j],[P_i])=0$,
  and $\lambda_n([P_i],[P_j])=-\delta_{ij}$.  Therefore $\ell_i =
  [S_i]+[P_i]$ and $d_i = [S_i]$ satisfy~(3).

  If $\psi(\cP^n \pi_1(V-\Delta))=\{e\}$, then the proof of
  \cite[Theorem~3.2]{Cha:2010-01} is carried out without any change,
  using (1), (2) and~(3), to show $\rhot(M(K),\phi)=0$.
\end{proof}

\bibliographystyle{amsalpha} \renewcommand{\MR}[1]{}

\bibliography{research}

\end{document}